\documentclass[preprint,onefignum,onetabnum]{siamart251104}

\usepackage{
	algorithm,
	algpseudocodex,
	amsmath,
	amssymb,
	blkarray,
	changepage,
	mathtools,
	mathrsfs,
	multicol,
	enumitem,
	hyperref,
	dsfont,
	graphicx,
	overpic,
	subcaption
}
\usepackage{color,soul}

\newcommand{\R}{\mathbb R}

\newcommand{\N}{\mathbb N}
\newcommand{\C}{\mathbb C}
\newcommand{\E}{\mathbb E}
\renewcommand{\P}{\mathbb P}

\newcommand{\1}{\mathbf 1}

\newcommand{\cA}{\mathcal A}

\newcommand{\cJ}{\mathcal J}

\newcommand{\cL}{\mathcal L}

\newcommand{\cN}{\mathcal N}

\newcommand{\cS}{\mathcal S}

\newcommand{\cX}{\mathcal X}
\newcommand{\cY}{\mathcal Y}

\newcommand{\sE}{\mathscr E}

\newcommand{\dd}{\mathrm{d}}

\renewcommand{\phi}{\varphi}

\newcommand{\Beta}{\mathrm{B}}
\newcommand{\Fr}{\mathrm{F}}

\newcommand{\MACG}{\mathrm{MACG}}

\newcommand{\Or}{\mathbb{O}}
\newcommand{\St}{\mathbb{V}}

\newcommand{\zFz}{{}_0F_0}
\newcommand{\oFz}{{}_1F_0}

\newcommand{\tFo}{{}_2F_1}

\DeclareMathOperator{\diag}{diag}

\DeclareMathOperator{\etr}{etr}

\DeclareMathOperator{\rank}{rank}

\DeclareMathOperator{\spn}{span}
\DeclareMathOperator{\tr}{tr}
\DeclareMathOperator{\Unif}{Unif}

\definecolor{myred}{HTML}{FF453A}
\definecolor{MLorange}{HTML}{D95319}
\definecolor{MLblue}{HTML}{0072BD}
\definecolor{MLgreen}{HTML}{77AC30}

\definecolor{rev}{HTML}{a45ee5}
\definecolor{rev0}{HTML}{00b7eb}
\definecolor{rev1}{HTML}{cb270f}
\definecolor{rev2}{HTML}{32cd32}

\newenvironment{textwide}
    {\begin{adjustwidth}{-\leftmargin+\itemindent}{}%
     \vspace{-1em}
    }{\end{adjustwidth}}

\newsiamremark{remark}{Remark}
\newsiamremark{hypothesis}{Hypothesis}
\crefname{hypothesis}{Hypothesis}{Hypotheses}
\newsiamthm{claim}{Claim}
\newsiamremark{fact}{Fact}
\crefname{fact}{Fact}{Facts}
\newsiamthm{conjecture}{Conjecture}

\crefname{section}{Section}{Sections}
\crefname{subsection}{Subsection}{Subsections}
\crefname{conjecture}{Conjecture}{Conjectures}

\headers{Beyond singular value gaps}{C. Wang and A. Townsend}

\title{Beyond singular value gaps in randomized subspace approximation\thanks{Submitted to the editors September 20, 2026.
\funding{The work of C.~W.\ was supported by the NSF GRFP under grant DGE-2139899, and by the NSF under grant DMS-1929284 while in residence at ICERM in Providence, RI, during the Numerical PDEs: Analysis, Algorithms, and Data Challenges semester program.  The work of A.~T.\ was supported by NSF CAREER (DMS-2045646) and by the Defense Advanced Research Projects Agency (DARPA) through The Right Space (TRS) Disruption Opportunity (DARPA-PA-24-04-07).}}}

\author{Christopher Wang\thanks{Department of Mathematics,  Cornell University, Ithaca, NY (\email{cyw33@cornell.edu}, \email{townsend@cornell.edu}).}
\and Alex Townsend\footnotemark[2]}

\ifpdf
\hypersetup{
	pdftitle={Beyond singular value gaps in randomized subspace approximation},
	pdfauthor={C. Wang and A. Townsend}
}
\fi
\ifpdf
  \DeclareGraphicsExtensions{.eps,.pdf,.png,.jpg}
\else
  \DeclareGraphicsExtensions{.eps}
\fi

\begin{document}
\maketitle

%% ------------------------------------------------------------------
%% ABSTRACT
%% ------------------------------------------------------------------
\begin{abstract}
The success of randomized range finders (RRFs) is typically analyzed via the singular value gaps of a target matrix $A$. In this work, we show that the so-called Frobenius singular value ratio provides a sharper analysis of an RRF's subspace quality under Gaussian sketching.  For any matrix $A$ and any integer $k\ge1$, we derive an exact formula for the cumulative distribution function of the largest principal angle between the $k$-dominant singular subspace of $A$ and the approximate RRF subspace, expressing it in terms of a hypergeometric function.  We obtain definitive probabilistic guarantees for RRFs that are often stronger than those obtained previously.
\end{abstract}


%% ------------------------------------------------------------------
%% END HEADER
%% ------------------------------------------------------------------

\section{Introduction}

The randomized range finder (RRF) is a central tool in numerical linear algebra for constructing low-dimensional subspace approximations of large matrices (see~\cref{alg:RRF}).  By projecting a matrix onto a randomly generated test subspace,  RRFs enable efficient and scalable algorithms for low-rank approximation,  randomized singular value decomposition (SVD),  and principal component analysis~\cite{Halko2011, Martinsson2011,Woolfe2008}.  Beyond these core tasks, randomized subspace approximation is used in CUR and interpolative decompositions~\cite{Armstrong2025,Drineas2012}, spectral clustering and graph embedding~\cite{Boutsidis2015}, data-oblivious preconditioners such as Blendenpik~\cite{Avron2010}, and eigensolvers based on contour integration and Krylov subspaces~\cite{Beyn2012, Horning2022,Tang2014}. Random sketches can also be used as initial guesses for Riemannian gradient descent on the Grassmannian manifold~\cite{Absil2004, Alimisis2024}. In all of these applications,  the quality of the resulting computation hinges on how accurately the randomized subspace captures a target invariant subspace of the underlying matrix.

\begin{algorithm}
\caption{Randomized range finder (RRF) with oversampling}\label{alg:RRF}
\begin{algorithmic}[1]
\Require Matrix $A\in\R^{m\times n}$, target rank $k\ge1$,  and oversampling $p\ge0$
\State Draw an $n\times(k+p)$ Gaussian matrix with i.i.d.\ entries from $\cN(0,1)$
\State Compute $Y=A\Omega$
\State Compute an orthonormal basis $\widetilde U_1$ for the column space of $Y$
\State \Return $\widetilde U_1\in\R^{m\times(k+p)}$
\end{algorithmic}
\end{algorithm}

We analyze the RRF's ability to recover the leading $k$-dimensional singular subspace of a matrix $A$ using a $(k+p)$-dimensional Gaussian sketch, in terms of the largest principal angle between the true and the approximate singular subspace (see~\cref{s:background} for a precise definition). Let $A\in\R^{m\times n}$ be an unknown matrix with an unknown SVD $A=U\Sigma V^\top$, where $N:=\min(m,n)$, $U\in\R^{m\times N}$ and $V\in\R^{n\times N}$ have orthonormal columns, and $\Sigma=\diag(\sigma_1,\dots,\sigma_N)\in\R^{N\times N}$ with $\sigma_1\ge\dots\ge\sigma_N\ge0$. We can partition the SVD of $A$ into its leading and tail singular values as follows: 
\begin{equation}\label{eq:A-partition}
\begin{blockarray}{ccc}
	& \\
	\begin{block}{c[cc]}
		A = & U_1 & U_2 \\
	\end{block}
\end{blockarray}
\begin{blockarray}{cccc}
	k & N-k & & \\
	\begin{block}{[cc][c]c}
		\Sigma_1 & 0 & V_1^\top & k \\
		0 & \Sigma_2 & V_2^\top & N-k \\
	\end{block}
\end{blockarray}
\end{equation}
for any $0\leq k\le N$. Note that such a partition may be non-unique when $\sigma_k=\sigma_{k+1}$; in this case we make an arbitrary choice for $U_1$ and $V_1$ that yields a valid SVD partition.

Most bounds on such singular subspace approximation errors depend on gaps between the singular values of $A$, which could be thought to be unavoidable due to the classical theorems of Davis--Kahan~\cite{Davis1970} and Wedin~\cite{Wedin1972} or the many estimates that depend on singular value gaps or a related quantity~\cite{Balcan2016,Massey2024,Massey2025,Nakatsukasa2020,Saibaba2019,Xu2020}. However, we find that a more relevant quantity determining the RRF's quality is instead the Frobenius singular value ratio, defined by 
\begin{equation}\label{eq:Fgap}
\xi_k(\Sigma) := \frac{\|\Sigma_1^{-1}\|_\Fr\|\Sigma_2\|_\Fr}{\sqrt{k(N-k)}} = \sqrt{\frac{\left(\sum_{j=1}^k\sigma_j^{-2}\right)\left(\sum_{j=k+1}^N\sigma_j^2\right)}{k(N-k)}}, \quad 1\leq k<N,
\end{equation}
where $\|\cdot\|_\Fr$ denotes the Frobenius norm. This quantity is well-defined whenever $\Sigma_1$ is non-singular, which holds in all practical settings of interest in which the target rank $k$ is smaller than the actual rank of the matrix.

One should view $\xi_k$ as a relaxation of the usual singular value ratio given by 
\begin{equation}\label{eq:gap}
\rho_k := \|\Sigma_1^{-1}\|\|\Sigma_2\| = \frac{\sigma_{k+1}}{\sigma_k}, \qquad 1\leq k<N,
\end{equation}
where $\|\cdot\|$ is the spectral norm; the additional $\sqrt{k(N-k)}$ factor in~\cref{eq:Fgap} is for norm equivalence. Indeed, the quantity $\xi_k^2$ is the ratio of the arithmetic mean of the squared tail singular values $\Sigma_2^2$, over the harmonic mean of the squared leading singular values $\Sigma_1^2$. Therefore, $\xi_k$ is strictly smaller than the singular value ratio $\rho_k$ whenever the diagonal entries of $\Sigma_1$ or $\Sigma_2$ exhibit any kind of decay. As a result, we obtain more informative estimates compared to, for instance,~\cite{Massey2024,Massey2025,Saibaba2019}, which do not provide meaningful bounds when the singular value gaps are small; see~\cref{eq:summary-bd}.  Our estimates help explain the strong empirical performance of the RRF for matrices with decaying singular values that have no large gaps between singular values.

The idea of bounding the approximation error in terms of the Frobenius norm of the tail singular values appeared earlier in~\cite{Dong2024}; our bounds improve on theirs by being simpler, with reasonable, explicit constants, and holding in all parameter regimes.

\subsection{Contributions}

Our analysis of the RRF can be split into two main parts: the first part is a highly technical analysis of the RRF approximation error through the lens of multivariate statistical analysis and random matrix theory, which we utilize in the second part to prove simple, practical bounds. For practitioners primarily interested in the latter, the important bound is given in~\cref{eq:summary-bd}.

Our first main contribution is the derivation of an exact, computable formula for the cumulative distribution function (CDF)\ of the angular error of the RRF's singular subspace approximation, in terms of the largest principal angle between the column spaces of $U_1$ and $\widetilde U_1$. Exactness here means that we completely characterize the subspace approximation error of the RRF, making our expressions quantitatively stronger than all existing estimates for subspace approximation, given knowledge of the singular values. In particular, our result improves on existing theory by holding in full generality: it requires no assumptions on the matrix $A$ nor on the target rank $k\ge1$ nor the oversampling parameter $p\ge0$; there are no asymptotic limits; and all constants are explicit (see~\cref{thm:cdf}).~\cref{fig:cdf-numerics} demonstrates the agreement of our CDF formula with empirical observations.\footnote{Code for computing the CDF can be found at \texttt{\url{https://github.com/chriswang030/RRFErrorCDF}}.}

\begin{figure}
\centering
\begin{minipage}{0.48\textwidth}
\begin{overpic}[width=\textwidth]{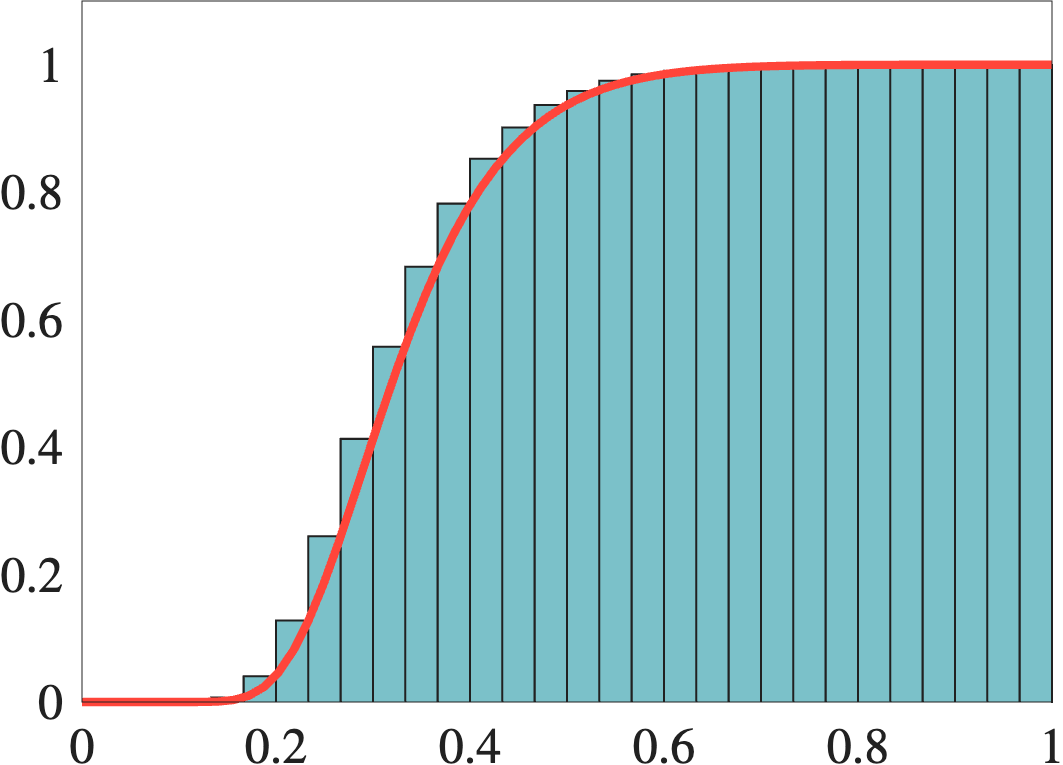}
\put(50,-6){\footnotesize $\sin\theta$}
\put(23,25){\rotatebox{69}{\footnotesize exact CDF~\cref{eq:cdf}}}
\end{overpic}
\vspace*{0mm}
\subcaption{Full-rank: $\sigma_j(A)=j^{-2}$, $1\le j\le100$.\qquad\qquad}
\end{minipage}
\hfill
\begin{minipage}{0.48\textwidth}
\begin{overpic}[width=\textwidth]{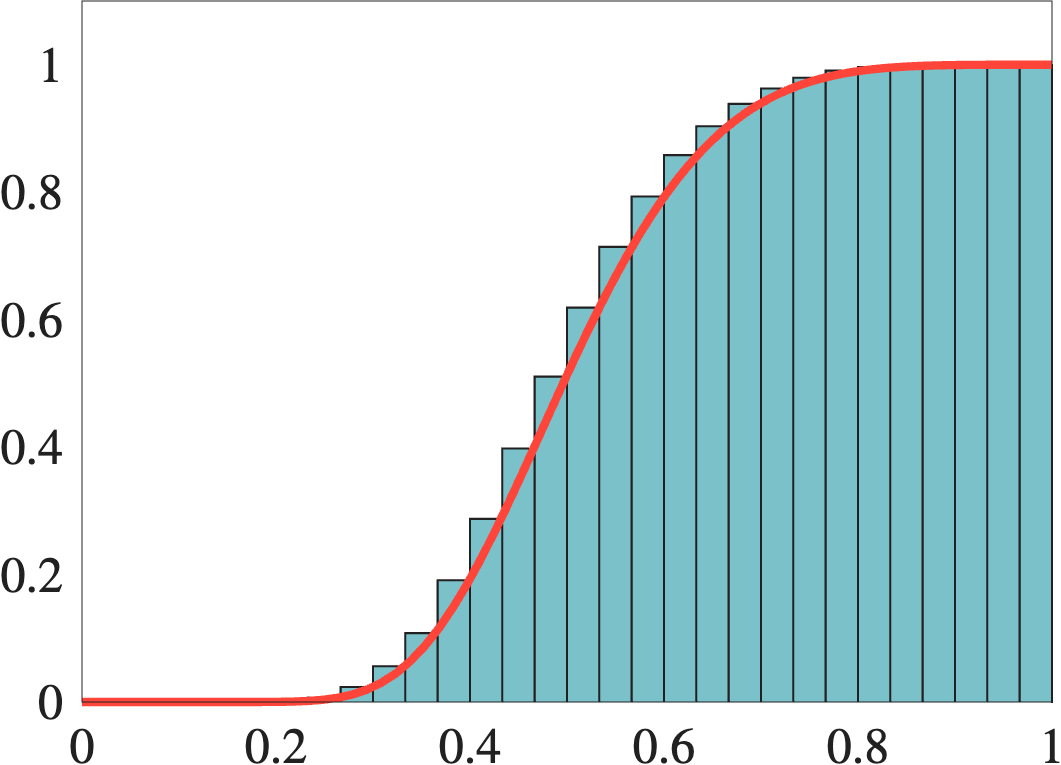}
\put(50,-6){\footnotesize $\sin\theta$}
\put(38,25){\rotatebox{65}{\footnotesize exact CDF~\cref{eq:cdf}}}
\end{overpic}
\vspace*{0mm}
\subcaption{Rank-deficient: $\sigma_j(A)=j^{-1}$, $1\le j\le2k+p-3$, and $\sigma_j(A)=0$, $j>2k+p-3$.}
\end{minipage}
\caption{The histograms show 10,000 samples of $\theta_1$ via the RRF (see~\cref{alg:RRF}) applied to a fixed matrix $A$, which can be singular, with $N=100$, $k=7$, and $p=7$. The solid red curves are computed from our formula for the CDF~\cref{eq:cdf} using the fast algorithm of~\cite{Koev2006} with 1,000 Monte Carlo iterations. Each curve was computed in $<1$ second in MATLAB 2026a on a standard 2024 Macbook Pro.}
\label{fig:cdf-numerics}
\end{figure}

This exact formula leads to our second main contribution, which is to derive simple error estimates for the RRF that indicate that the Frobenius singular value ratio $\xi_k$ is a better gauge for the quality of the RRF approximation compared to the singular value gap $\rho_k$. Our result shows that existing error estimates for the RRF stated in terms of $\rho_k$, such as those in~\cite{Saibaba2019}, can be restated essentially by replacing the singular value ratio with the corresponding Frobenius singular value ratio. For example, we show that by running the RRF on a matrix $A\in\R^{m\times n}$ with target rank $k\ge1$, oversampling $p\ge1$, and $N=\min(m,n)\ge 2k+p$, we have with probability $\ge1-\delta$ that
\begin{equation}\label{eq:summary-bd}
\sin\theta_1 \le \frac{\xi_k C_0}{\sqrt{1+\xi_k^2 C_0^2}}
\le \frac{\frac{1}{\sqrt{2\delta}}\|\Sigma_1^{-1}\|_\Fr\|\Sigma_2\|_\Fr}{\sqrt{1+\frac{1}{2\delta}\|\Sigma_1^{-1}\|_\Fr^2\|\Sigma_2\|_\Fr^2}}.
\end{equation}
where $\theta_1$ is the largest principal angle between the dominant left $k$-singular subspace of $A$ and the RRF approximation, $\xi_k$ is the Frobenius singular value ratio~\cref{eq:Fgap}, and $C_0=\sqrt{k(N-k-1)/(2\delta)}$. Similar bounds also hold for $p=0$ and/or $N<2k+p$ (see~\cref{thm:Fgap-bound}). As a corollary, we obtain a bound in terms of the singular value ratio $\rho_k$ comparable to~\cite[Thm.~6]{Saibaba2019}, with constants that can be much better for small $p$ (see~\cref{thm:gap-bound}).

\subsection{Outline}
The structure of the paper is as follows. In~\cref{s:background}, we give background information. We state and prove our exact formula for the distribution function of the RRF error in~\cref{s:formula}.~\cref{s:applications} establishes our main error estimate~\cref{eq:summary-bd} as well as a version in terms of the singular value ratio $\rho_k$ with improved constants. We also provide numerical experiments comparing our estimates against the true approximation error as well as the state-of-the-art estimates of~\cite{Saibaba2019}. Finally, we discuss a conjectured improvement of our estimates, practical alternatives for estimating upper bounds on the error, and extensions in~\cref{s:discussion}.

\section{Background}\label{s:background}

We briefly provide background information on the mathematical tools used throughout the work.

\subsection{Principal angles}
Principal angles are generalizations of the usual notion of angle between two lines to higher dimensional spaces. Any $k$-dimensional subspace of $\R^n$ can be identified with a matrix $X\in\R^{n\times k}$ whose columns are orthonormal and span that subspace. For a second subspace in $\R^n$, possibly of different dimension, identified with $Y\in\R^{n\times\ell}$, we define the principal angles $\pi/2\ge\theta_1(X,Y)\ge\dots\ge\theta_{\min(k,\ell)}(X,Y)\ge0$ between $X$ and $Y$ by $\cos\theta_j(X,Y) = \sigma_{\min(k,\ell)-j+1}(X^\top Y),$ where $\sigma_1(X^\top Y)\ge\dots\ge\sigma_{\min(k,\ell)}(X^\top Y)\ge0$ are the singular values of $X^\top Y$. We use $\Theta(X,Y)$ to denote the diagonal matrix of principal angles between $X$ and $Y$~\cite{Bhatia1997,Golub2013}.

\subsection{Matrix spaces and probability measures}
All matrices are real unless otherwise stated. For positive integers $k,\ell,r$ with $r\le\min(k,\ell)$, let $\cL_{k,\ell}^+(r)$ be the space of $k\times\ell$ matrices of rank $r$. Let $\cS_k^+(r)\subset\cL_{k,k}^+(r)$ denote the subspace of symmetric positive semidefinite (SPSD) matrices of rank $r$ with distinct non-zero eigenvalues.\footnote{\,The requirement of having distinct eigenvalues poses no issue in our analysis since the set of matrices with repeated eigenvalues has measure zero.} Let $\cL_{k,\ell}^{++}$ and $\cS_k^{++}$ denote the respective subspaces of full-rank matrices, i.e., $\cS_k^{++}$ denotes the symmetric positive definite (SPD) matrices with distinct eigenvalues. For $\ell\ge k$, let $\St_{\ell,k}$ be the Stiefel manifold of $\ell\times k$ matrices with orthonormal columns, and let $\Or_k$ be the manifold of $k\times k$ orthogonal matrices. 

For symmetric $X,Y$, we use the L\"owner ordering, i.e., we write $X\succ(\prec)\ Y$ to mean that $X-Y$ is positive (negative) definite. Similarly, $X\succeq(\preceq)\ Y$ means that $X-Y$ is positive (negative) semidefinite. As shorthand, we sometimes write $X\succ t$ for some $t\in\R$ to mean $X\succ tI$.

To derive our formulas, we integrate over several matrix manifolds with respect to measures that we specify here; see~\cite{Chikuse2003, DiazGarcia2007, Muirhead1982, Uhlig1994} for details. Measures on matrix manifolds are always written in parentheses, e.g., $(\dd X)$. On $\R^{\ell\times k}$, $\cS_k^{++}$, and $\cL_{k,\ell}^{++}$, $(\dd X)$ is the standard Lebesgue measure with respect to the functionally independent entries; for $\R^{\ell\times k}$ and $\cL_{k,\ell}^{++}$, these are all $k\ell$ entries of the matrix, whereas for $\cS_k^{++}$, the functionally independent entries are the $k(k+1)/2$ entries on or above the main diagonal. On $\Or_n$ and $\St_{n,k}$, $(\dd X)$ is the unique probability measure that is both left and right invariant under orthogonal transformations, namely, the transformations
\begin{align*}
H \mapsto QHR\in\Or_n,& \qquad\qquad H,Q,R\in\Or_n, \\
H_1 \mapsto Q_1H_1R_1\in\St_{n,k},& \qquad\qquad H_1\in\St_{n,k},\,  Q_1\in\Or_n,\,R_1\in\Or_k.
\end{align*}
These measures are known as the normalized Haar measures; we refer to them as the uniform probability measures on $\Or_n$ and $\St_{n,k}$.\footnote{\,Existence and uniqueness of the Haar measure, up to a constant scalar multiple, is well-known for compact topological groups and homogeneous spaces, such as $\Or_n$ and $\St_{n,k}$. For a specific reference, see Section III.4 of~\cite{Nachbin1965}, in particular Example 4 concerning Stiefel manifolds.} 

On $\cS_k^+(r)$, with $r<k$, $(\dd X)$ is the measure invariant under conjugation by orthogonal matrices, namely $X\mapsto QXQ^\top$ for $Q\in\Or_k$, scaled in agreement with~\cite{Uhlig1994}. A natural coordinate system for this measure is via the entries of its eigenvalues and eigenvectors: for each $X\in\cS_k^+(r)$, we may write its economy-sized eigendecomposition $X=H_1DH_1^\top$, where the columns of $H_1\in\St_{k,r}$ are the eigenvectors, and $D$ is an $r\times r$ diagonal matrix of eigenvalues. Then $X$ is uniquely characterized by $H_1$ and $D$, up to a choice of signs for each eigenvector. An explicit representation of the measure in terms of these coordinates is given in~\cite[Thm.~2]{Uhlig1994}.

\subsection{Zonal polynomials and hypergeometric functions of matrix argument}\label{ss:bg-matrix}

Extremely important in our proofs are the hypergeometric functions of matrix argument; for details, see~\cite[Ch.~7]{Muirhead1982} and~\cite[Ch.~35]{NIST}. An integer partition $\kappa\,\vdash\ell$ for some integer $\ell\ge0$ is a decreasing sequence of non-negative integers whose sum is $\ell$, and $|\kappa|:=\ell$ is called the weight of $\kappa$. The sum over all integer partitions is denoted by  $\sum_\kappa$. For a square matrix $X$, $C_\kappa(X)$ denotes the zonal polynomial of the eigenvalues of $X$ indexed by a partition $\kappa$; they are homogeneous, symmetric, degree-$|\kappa|$ polynomials and can be defined as the eigenfunctions of the Laplace--Beltrami operator on spaces of SPD matrices.\footnote{\,Zonal polynomials can be defined in many equivalent ways: recursively as orthogonal polynomials, combinatorially, geometrically, or representation theoretically. They are related to Schur polynomials and are a special case of the Jack and Macdonald polynomials; see, e.g.,~\cite{James1968,Jiu2020,Macdonald1995}.} We normalize according to the convention $\sum_{\kappa\,\vdash\ell} C_\kappa(X) = (\tr X)^\ell$ for each integer $\ell\ge0$. Typically, zonal polynomials are defined for symmetric matrices. However, since zonal polynomials depend only on the eigenvalues of the matrix argument, and since the eigenvalues of $AB$ and $BA$ are the same for two square matrices $A,B$, we extensively utilize the identity $C(AB)=C(A^{1/2}BA^{1/2})$, where the square root denotes the unique SPSD square root whenever well-defined.

For integers $m,n\ge0$ and parameters $a_1,\dots,a_m,b_1,\dots,b_n\in\C$, we denote by ${}_mF_n(a_1,\dots,a_m;b_1,\dots,b_n;\,\cdot\,)$ the hypergeometric functions of either one or two matrix arguments, depending on context. For $k\times k$ matrices $X,Y$, the hypergeometric functions can be expressed as series of zonal polynomials:
\begin{equation}\label{eq:pFq}
{}_mF_n(a_1,\dots,a_m;b_1,\dots,b_n;X,Y) = \sum_\kappa \frac{(a_1)_\kappa\cdots(a_m)_\kappa}{|\kappa|!(b_1)_\kappa\cdots(b_n)_\kappa}\frac{C_\kappa(X)C_\kappa(Y)}{C_\kappa(I_k)}.
\end{equation}
Here, $(a)_\kappa$ is the partitional shifted factorial, defined by $(a)_\kappa = \prod_{i=1}^\ell \left(a - \tfrac12(i-1)\right)_{\kappa_i}$ with respect to a partition $\kappa=(\kappa_1,\dots,\kappa_\ell)$, where $(a)_{\kappa_i} = a(a+1)\cdots(a+\kappa_i-1)$ is the rising factorial. When only a single matrix argument $X$ appears, the hypergeometric function is defined as in~\cref{eq:pFq} with $Y=I_k$. Our hypergeometric functions may also take two matrix arguments of different sizes~\cite{Shimizu2021}, in which case a superscript such as ${}_mF_n^{(k)}$ denotes the size of the identity matrix appearing in the denominator of~\cref{eq:pFq}. 

Certain familiar matrix functions can be expressed as hypergeometric series. We make extensive use of the identities $\zFz(X) = \etr(X)$ and $\oFz(a;X) = |I-X|^{-a}$, where $|\cdot|$ denotes the determinant, $\tr(\cdot)$ denotes the trace, and $\etr(\cdot):=e^{\tr(\cdot)}$~\cite[Eqs.~(7.3.3--4)]{Muirhead1982}. Note that the latter identity only holds when $\|X\|<1$, as otherwise the hypergeometric series does not converge. For singular matrices, we use $|\cdot|_+$ to denote the pseudodeterminant, i.e., the product of the non-zero eigenvalues. 

Lastly, two related functions that occasionally appear are the multivariate gamma and beta functions $\Gamma_k(a)$ and $\Beta_k(a,b)$. They are defined as
\[
\Gamma_k(a) := \pi^{\tfrac{k(k-1)}4}\prod_{j=1}^k \Gamma\left(a-\frac{j-1}2\right), \qquad\Beta_k(a,b) := \frac{\Gamma_k(a)\Gamma_k(b)}{\Gamma_k(a+b)}
\]
for $a,b\in\C$ satisfying $\Re(a),\Re(b)>\tfrac{k-1}2$~\cite[Ch.~35.3]{NIST}, where $\Gamma$ is the standard gamma function.

\section{Distribution function of the RRF approximation error}\label{s:formula}

In this section, we derive the cumulative distribution function (CDF)\ of the largest principal angle between the exact dominant singular subspace,  $U_1$,  and the one computed by the RRF, $\widetilde U_1$. Afterward, we provide heuristics to help make sense of the formula as well as numerical experiments verifying its correctness.

\begin{theorem}\label{thm:cdf}
Run the RRF (\cref{alg:RRF}) on a matrix $A\in\R^{m\times n}$ with $k\ge1$, $p\ge0$ and $N=\min(m,n)\ge k+p$. Under partition~\cref{eq:A-partition}, let $\theta_1$ be the largest principal angle between the true $k$-dominant left singular subspace and the RRF approximation.
\begin{enumerate}[leftmargin=8mm]
\item If $k+p\ge\rank(A)$, then $\theta_1=0$ almost surely.
\item If $k+p<\rank(A)$ and $N\ge 2k+p$, then for $0\le\theta\le\frac\pi2$, we have
\begin{textwide}
\begin{multline}\label{eq:cdf}
\Phi_{\Sigma,k,p}(\theta) := \P(\theta_1<\theta) \\
= K_{N,k,p}\ \E\left[|S_\theta|^{\frac{N-k-p}2} \tFo(\tfrac{-p+1}2, \tfrac{N-k-p}2; 
	\tfrac{N-p+1}2; S_\theta) \right],
\end{multline}
\end{textwide}
where $K_{N,k,p} := \frac{\Gamma_k(\frac N2)\Gamma_k(\frac{k+1}2)}{\Gamma_k(\frac{k+p}2)\Gamma_k(\frac{N-p+1}2)}$ and
\begin{textwide}
\begin{equation}\label{eq:S}
	S_\theta \equiv S_\theta(\Sigma) := (I_k+\cot^2(\theta)\Sigma_1^{-1}(P_\perp Q_1^\top(H_1^\top\Sigma_2^2H_1)^\dagger Q_1P_\perp)^\dagger\Sigma_1^{-1})^{-1}
\end{equation}
\end{textwide}
is the $k\times k$ random matrix in which $\dagger$ denotes the Moore--Penrose pseudoinverse, $Q_1\sim\Unif(\St_{k+p,k})$ and $H_1\sim\Unif(\St_{N-k,k+p})$ are independent random matrices, and $P_\perp\in\R^{k\times k}$ is the orthogonal projection onto the orthogonal complement of $Q_1^\top\mathrm{null}(H_1^\top\Sigma_2^2H_1) := \spn\{Q_1^\top v:v\in\mathrm{null}(H_1^\top\Sigma_2^2H_1)\}$ with $\mathrm{null}(\cdot)$ denoting the nullspace. In particular, when $r\ge 2k+p$, $P_\perp=I_k$.
\item If $k+p<\rank(A)$ and $N<2k+p$, then~\cref{eq:cdf} and~\cref{eq:S} hold after replacing $N$ everywhere with $2k+p$. In particular,
\[
\Phi_{\Sigma,k,p}(\theta) 
= K_{2k+p,k,p}\ \E\left[|\widehat S_\theta|^{\frac k2}\tFo(\tfrac{-p+1}2,\tfrac k2;\tfrac{2k+1}2;\widehat S_\theta)\right]
\]
where
\[
\widehat S_\theta = (I_k+\cot^2(\theta)\Sigma_1^{-1}(\widehat P_\perp Q_1^\top(\widehat H_1^\top\widehat\Sigma_2^2 \widehat H_1)^\dagger Q_1\widehat P_\perp)^\dagger\Sigma_1^{-1})^{-1}
\]
with $\widehat\Sigma_2 := \diag(\sigma_{k+1},\dots,\sigma_N,0,\dots,0)\in\R^{(k+p)\times(k+p)}$, $\widehat H_1\sim\Unif(\Or_{k+p})$, and $\widehat P_\perp$ the projection onto the orthogonal complement of $Q_1^\top\mathrm{null}(\widehat H_1^\top\widehat\Sigma_2^2\widehat H_1)$.
\end{enumerate}
\end{theorem}

Our proof of~\cref{thm:cdf} is in the spirit of~\cite{Absil2006} and contains their result as a special case. The main idea is that the vector of squared cosines of the principal angles between the left singular subspaces of a given matrix $A$ versus a Gaussian sketch $A\Omega$, can be represented stochastically as the eigenvalues of a certain  random matrix closely resembling a so-called beta random matrix. Beta random matrices can be represented as a ratio $(W_1+W_2)^{-1/2}W_1(W_1+W_2)^{-1/2}$, where the square root denotes the unique SPSD square root and $W_1,W_2$ are Wishart random matrices, which themselves have the representation $W_j=\Omega_j^\top\Omega_j$ for tall-skinny Gaussian matrices $\Omega_j$. We require two generalizations: first, the Wishart matrices are replaced by matrix quadratic forms, and second, they are allowed to be singular. In other words, we are interested in the eigenvalues of a matrix $(Z_1+Z_2)^{-1/2}Z_1(Z_1+Z_2)^{-1/2}$ where $Z_j = \Omega_j^\top\Sigma_j\Omega_j$ for some SPSD matrices $\Sigma_j$. After deriving the distribution of this generalized beta random matrix, we obtain the largest principal angle of interest by extracting its smallest non-zero eigenvalue. We welcome most readers to jump the following technical proof and go straight to~\cref{s:applications}.

We first establish a few tools to help deal with probability distributions on spaces of singular random matrices, which in general form a measure zero subset of more familiar matrix spaces. In what follows, for measurable spaces $\cX,\cY$, the pushforward of a measure $\mu$ on $\cX$ with respect to a measurable map $f:\cX\to\cY$ is defined as the measure $f_*\mu$ on $\cY$ given by $(f_*\mu)(K) = \mu(f^{-1}(K))$, for all measurable sets $K\subset\cY$. We begin with a lemma generalizing the so-called Wishart integral to the singular setting~\cite{Hsu1940}.

\begin{lemma}\label{lem:wishart-integral}
Fix positive integers $k<\ell$. Let $\phi$ be a density function with respect to the Lebesgue measure on $\cL_{k,\ell}^{++}$ such that $\phi(X)$ only depends on the entries of $X^\top X$, i.e., $\phi(X) = \psi(X^\top X)$ for some function $\psi:\cS_\ell^+(k)\to\R$. Let $s:\cL_{k,\ell}^{++}\to\cS_\ell^+(k)$ be the map $s(X)=X^\top X$. Then the pushforward density $s_*\phi$ on $\cS_\ell^+(k)$ is
\begin{equation}\label{eq:wishart-integral}
	(s_*\phi)(A) = \frac{\pi^{k^2/2}}{\Gamma_k(\frac k2)} |A|_+^{-\frac{\ell-k+1}2}\psi(A), \qquad A\in\cS_\ell^+(k).
\end{equation}
\end{lemma}
\begin{proof}
Let $X\in\cL_{k,\ell}^{++}$ be a random matrix with density $\phi$ with economy-sized SVD $X=QDH_1^\top$, where $Q\in\Or_k$ and $H_1\in\St_{\ell,k}$. Also let $A=X^\top X=H_1D^2H_1^\top$. By \cite[Thm.~3.2]{DiazGarcia1997b},\footnote{\label{fn:norm}\,The result in \cite{DiazGarcia1997b} differs from ours by a factor of $\frac{2^k\pi^{k^2/2}}{\Gamma_k(k/2)}$ because they use a different normalization of the Haar measure on $\Or_k$. The constant factor is the volume of $\Or_k$ with respect to their normalization and is given in, e.g.,~\cite[Cor.~2.1.16]{Muirhead1982}.} the Jacobian of $X\mapsto(A,Q)$ is given by
\[
(\dd X) = \frac{\pi^{k^2/2}}{\Gamma_k(\frac k2)} |A|_+^{-(\ell-k+1)/2}(\dd A)(\dd Q).
\]
For any Lebesgue measurable set $K\subset\cS_\ell^+(k)$, we have by invariance of the integrand with respect to $Q$ and the fact that $\int_{\Or_k}(\dd Q)=1$, by our choice of normalization, that
\begin{align*}
	\P(X^\top X\in K)
	&= \int_{\cL_{k,\ell}^{++}} \1\{X^\top X\in K\}\,\phi(X)\,(\dd X) \\
	&= \frac{\pi^{k^2/2}}{\Gamma_k(\frac k2)} \int_{\Or_k} \int_{\cS_\ell^+(k)} \1\{A\in K\} |A|_+^{-\frac{\ell-k+1}2} \psi(A)\,(\dd A)\,(\dd Q) \\
	&= \frac{\pi^{k^2/2}}{\Gamma_k(\frac k2)} \int_{A\in K} |A|_+^{-\frac{\ell-k+1}2}
		\psi(A)\,(\dd A),
\end{align*}
where $\1\{\cdot\}$ denotes the indicator function of the measurable set $\{\cdot\}$, and the result follows.
\end{proof}

The singular Wishart integral derived above allows us now to derive the density function of singular  quadratic forms of Gaussian random matrices.

\begin{lemma}\label{lem:quadratic-form}
Let $k<\ell$ be positive integers and let $B\in\cS_k^{++}$. If $\Omega\sim\cN_{k,\ell}(0,1)$ is a $k\times\ell$ random matrix with i.i.d.~standard Gaussian entries, then the density of $\Omega^\top B\Omega$ in $\cS_\ell^+(k)$ is
\begin{equation}\label{eq:quadratic-form}
	\phi(A) = \frac{2^{-k\ell/2}\pi^{-k(\ell-k)/2}}{\Gamma_k(\frac k2)} |B|^{-\frac \ell2} 
	|A|_+^{-\frac{\ell-k+1}2} \zFz^{(k)}(B^{-1}, -\tfrac12 A).
\end{equation}
\end{lemma}
\begin{proof}
By~\cite[Thm.~3.1.1]{Muirhead1982}, $\Omega$ is almost surely full rank and has density $\phi(\Omega) = (2\pi)^{-k\ell/2} \etr(-\tfrac12 \Omega^\top \Omega)$. For each $Q\in\Or_k$, set $Y\equiv Y(Q) := Q^\top B^{1/2}\Omega$ with $(\dd Y)=|B|^{\frac \ell2}(\dd \Omega)$ and density $\phi(Y) = (2\pi)^{-k\ell/2} \etr(-\tfrac12 Y^\top Q^\top B^{-1}QY) |B|^{-\frac \ell2}$. Observe that $\Omega^\top B\Omega=Y^\top Y$ for all $Q\in\Or_k$. Then for any Lebesgue measurable set $K\subset\cS_\ell^+(k)$, we have
\begin{align*}
	\P(\Omega^\top &B\Omega\in K)
	= (2\pi)^{-k\ell/2}\int_{\R^{k\times\ell}} \1\{\Omega^\top B\Omega\in K\}\etr(-\tfrac12\Omega^\top\Omega)\,(\dd\Omega) \\
	&= (2\pi)^{-k\ell/2} |B|^{-\frac \ell2} \!\!\int_{\Or_k}\,\int_{\cL_{k,\ell}^{++}}\1\{Y^\top Y\in K\}\etr(-\tfrac12 Y^\top Q^\top B^{-1}QY) \,(\dd Y)\,(\dd Q) \\
	&= \int_{Y^\top Y\in K} (2\pi)^{-k\ell/2} |B|^{-\frac \ell2} \zFz^{(k)}(B^{-1}, -\tfrac12 Y^\top Y)\,(\dd Y)
\end{align*}
where we use that $\int_{\Or_k}(\dd Q)=1$, the identity $\zFz(\cdot)=\etr(\cdot)$, and~\cite[Thm.~7.3.3]{Muirhead1982}; interchange of integrals is justified by Tonelli's theorem since the integrands are non-negative. The integrand is a density function on $\cL_{k,\ell}^{++}$ which depends only on $Y^\top Y$, so pushing forward with respect to the map $s(Y)=Y^\top Y$ of~\cref{lem:wishart-integral} yields
\[
	\P(\Omega^\top B\Omega\in K) = \frac{2^{-k\ell/2}\pi^{-k(\ell-k)/2}}{\Gamma_k(\frac k2)} |B|^{-\frac \ell2} \int_{A\in K} |A|_+^{-\frac{\ell-k+1}2}\zFz^{(k)}(B^{-1}, -\tfrac12 A)\,(\dd A),
\]
integrating over a subset of $\cS_\ell^+(k)$. The result follows by definition.
\end{proof}

The density~\cref{eq:quadratic-form} is analogous to the density derived by similar methods in~\cite{Hayakawa1966} for non-singular quadratic forms; the presence of the $\pi^{-k(\ell-k)/2}$ factor, as well as the pseudodeterminant, can be explained by the fact that the density is defined on a measure-zero subset of $\cS_\ell^{++}$ of singular matrices.

Our final lemma is a computation that will used to simplify a later computation.

\begin{lemma}\label{lem:partial-integral}
Fix $X\in\cS_\ell^{++}$, $a\in\C$ with $\Re(a)>\tfrac{\ell-1}2$, and a partition $\kappa$. Set $\nu=\tfrac{\ell+1}2$. Then
\begin{equation}\label{eq:partial-integral}
	\int_{0\prec U\prec X} |U|^{a-\nu} C_\kappa(-U)\,(\dd U) 
		= \frac{(a)_\kappa}{(a+\nu)_\kappa} \Beta_\ell(a,\nu) 
		|X|^a C_\kappa(-X),
\end{equation}
where the integral is over all $U\in\cS^{++}_\ell$ such that $U\prec X$.
\end{lemma}
\begin{proof}
Set $f_1(U) = |U|^{a-\nu}C_\kappa(-U)$ and $f_2(U)=1$. Write the integral on the left-hand side of~\cref{eq:partial-integral} as $f(X) = \int_{0\prec U\prec X} f_1(U)f_2(X-U)\,(\dd U)$. Let $\cL$ denote the matrix-variate Laplace transform, so that $\cL f_1(Z) = (a)_\kappa\Gamma_\ell(a)|Z|^{-a}C_\kappa(-Z^{-1})$ and $\cL f_2(Z) = \Gamma_\ell(\nu)|Z|^{-\nu}$ by~\cite[Thms.~2.1.11 and 7.2.7]{Muirhead1982}. By the convolution theorem~\cite[Prob.~7.3]{Muirhead1982}, we have $\cL f(Z) = (a)_\kappa\Gamma_\ell(a)\Gamma_\ell(\nu)|Z|^{-(a+\nu)}C_\kappa(-Z^{-1}),$ which is also the Laplace transform of the right-hand side of~\cref{eq:partial-integral} as a function of $X$. We conclude by uniqueness of the Laplace transform.\footnote{\,The matrix-variate Laplace transform is nothing more than the standard Euclidean Laplace transform in the functionally independent entries of the matrix \cite[p.~479]{Herz1955}. Uniqueness of the Laplace transform is thus given by the classical theory, e.g., the inversion theorem.}
\end{proof}

We are now ready to prove~\cref{thm:cdf}, which essentially consists of finding the CDF of the smallest non-zero eigenvalue of a generalized beta random matrix defined in terms of singular Gaussian quadratic forms instead of Wishart matrices.

\begin{proof}[Proof of~\cref{thm:cdf}]
The case where $k+p\ge r:=\rank(A)$ is trivial. For the $k+p<r$ case, we split the proof into two parts: the full-rank regime with sufficiently large $N$ ($r=N\ge 2k+p$), and the rank-deficient regime ($k+p<r<2k+p$).

{\bf Part I: Full-rank regime.} Assume that $A$ is full-rank and $r=N\ge 2k+p$. We partition $A=U\Sigma V^\top$ as in~\cref{eq:A-partition}, and let $\widetilde U_1\in\R^{m\times(k+p)}$ be the output of the RRF. Recall that $\Omega\sim\cN_{n,k+p}(0,1)$ is rotationally invariant, so that $V^\top\Omega\sim\cN_{N,k+p}(0,1)$. Moreover, $\Theta(U_1,\widetilde U_1) = \Theta(U^\top U_1,U^\top \widetilde U_1)$. Therefore, we may assume without loss of generality that our matrix of interest $A$ is simply the $N\times N$ diagonal matrix $\Sigma$ of its non-zero singular values and that $\Omega\sim\cN_{N,k+p}(0,1)$. Let us conformally partition $\Omega = \begin{bmatrix} \Omega_1 \\ \Omega_2 \end{bmatrix}$ so that $\Omega_1\sim\cN_{k,k+p}(0,1)$ and $\Omega_2\sim\cN_{N-k,k+p}(0,1)$ are independent. The column space of $\Sigma\Omega$ is $\widetilde U_1$; an orthonormal basis is given by the columns of
\[
	\widetilde U_1 := \Sigma \Omega(\Omega^\top \Sigma^2\Omega)^{-1/2} = \begin{bmatrix} \Sigma_1\Omega_1 \\ \Sigma_2\Omega_2 \end{bmatrix} (\Omega_1^\top \Sigma_1^2\Omega_1 + \Omega_2^\top \Sigma_2^2\Omega_2)^{-1/2}.
\]
Without loss of generality, an orthonormal basis for $U_1$ is simply $I_{N,k}$, the $N\times k$ matrix with ones along the main diagonal and zeros elsewhere. The cosines of the principal angles between $U_1$ and $\widetilde U_1$ are the non-zero singular values of 
\begin{equation}\label{eq:C-sv}
C := I_{N,k}^\top \widetilde U_1 = \Sigma_1\Omega_1(\Omega_1^\top \Sigma_1^2\Omega_1 + \Omega_2^\top \Sigma_2^2\Omega_2)^{-1/2}.
\end{equation}
Thus, our task is to derive a joint density function for the eigenvalues of 
\begin{equation}\label{eq:C'C}
C^\top C = (\Omega_1^\top \Sigma_1^2\Omega_1 + \Omega_2^\top \Sigma_2^2\Omega_2)^{-1/2}\Omega_1^\top \Sigma_1^2\Omega_1(\Omega_1^\top \Sigma_1^2\Omega_1 + \Omega_2^\top \Sigma_2^2\Omega_2)^{-1/2}.
\end{equation}
Since the classical beta type I ensemble arises when $X$ and $Y$ are independent Wishart matrices with identity scale,  we find that $C^\top C$ is a generalized beta-type random matrix; if $\Sigma_1,\Sigma_2$ were both identity matrices, then indeed $C^\top C$ would be a beta type I random matrix, in the sense of~\cite[Def.~5.2.1]{Gupta2000}.\footnote{\,This distribution is also known by the Jacobi orthogonal ensemble in the mathematical physics literature~\cite[Ch.~3.6]{Forrester2010}, and by the MANOVA ensemble in the statistics literature~\cite[Ch.~10]{Muirhead1982}.}

{\bf Step 1: Derive the density of $C^\top C$ near zero.} We begin by deriving a density function for the matrix $C^\top C$ of~\cref{eq:C'C} that is valid whenever its eigenvalues are near zero. Since $C^\top C$ is defined in terms of the matrix quadratic forms $\Omega_1^\top\Sigma_1^2\Omega_1$, $\Omega_2^\top\Sigma_2^2\Omega_2$, we first determine their respective density functions. The density of $\Omega_2^\top\Sigma_2^2\Omega_2$ is given by~\cite[Thm.~1]{Hayakawa1966} since $\Sigma_2^2$ is full-rank and has size $N-k\ge k+p$. On the other hand, since $\Omega_1^\top\Sigma_1^2\Omega_1$ is singular for $p>0$, its density is given in~\cref{lem:quadratic-form}.

We proceed to derive the density for $C^\top C$. In what follows, $\phi(\cdot)$ always denotes the density of a random quantity determined from context. Set $X=\Omega_1^\top \Sigma_1^2\Omega_1\in\cS_{k+p}^+(k)$ and $Y=\Omega_2^\top \Sigma_2^2\Omega_2\in\cS_{k+p}^{++}$. Their densities are given by~\cref{lem:quadratic-form} and~\cite[Thm.~1]{Hayakawa1966} respectively; by independence, their joint density is
\begin{multline*}
\phi(X,Y) = \frac{2^{-(k+p)N/2}\pi^{-kp/2}}{\Gamma_k(\frac k2)\Gamma_{k+p}(\frac{N-k}2)|\Sigma|^{k+p}} |X|_+^{-\frac{p+1}2} |Y|^{\frac{N-2k-p-1}2} \\
\cdot \zFz^{(k)}(\Sigma_1^{-2}, -\tfrac12 X) \,\zFz^{(N-k)}(\Sigma_2^{-2}, -\tfrac12 Y).
\end{multline*}
We transform $U=X+Y\in\cS_{k+p}^{++}$ and $V=U^{-1/2}XU^{-1/2}\in\cS_{k+p}^+(k)$. The Jacobian of this transformation is computed in two steps: first, we transform $(X,Y)\to(X,U)$, which has a Jacobian of 1 since $U$ is a linear function of $X,Y$. Then, fixing $U$, we transform $V=U^{-1/2}XU^{-1/2}$, whose Jacobian is given by $(\dd V) = |U|^{-\frac k2}|V|_+^{\frac{p+1}2}|X|_+^{-\frac{p+1}2} (\dd X)$, following~\cite[Thm.~1]{DiazGarcia1997}.\footnote{\,There is a typo in~\cite[Thm.~1]{DiazGarcia1997}; a negative sign should appear in the exponent of $|L|$, which can be easily verified using their equation (3). This error was corrected in \cite[Thm.~2.1]{DiazGarcia2009}.

In our application of~\cite[Thm.~1]{DiazGarcia1997}, we set their $Y$ to be our $V$, their $B$ to be our $U^{1/2}$, and keeping $X$ the same. Note that $K,L$ in their theorem are the diagonal matrices of the non-zero eigenvalues of our $X,V$ respectively, hence the appearance of pseudodeterminants in our derivation.} Since $X=U^{1/2}VU^{1/2}$ and $Y=U^{1/2}(I-V)U^{1/2}$, we conclude that the joint density of $U$ and $V$ is
\begin{multline}\label{eq:dens-UV}
	\phi(U,V) = \frac{2^{-(k+p)N/2}\pi^{-kp/2}}{\Gamma_k(\frac k2)\Gamma_{k+p}(\frac{N-k}2)
		|\Sigma|^{k+p}}|U|^{\frac{N-k-p-1}2}|V|_+^{-\frac{p+1}2}|I-V|^{\frac{N-2k-p-1}2} \\
	\cdot\,\zFz^{(k)}(\Sigma_1^{-2}, -\tfrac12 UV)\,\zFz^{(N-k)}(\Sigma_2^{-2}, -\tfrac12 U(I-V)),
\end{multline}
recalling that $UV$ and $U^{1/2}VU^{1/2}$ have the same eigenvalues (analogous for $U(I-V)$) and that hypergeometric functions only depend on the eigenvalues. To obtain the density of $V$, we integrate over $U\in\cS_{k+p}^{++}$. For $\|V\|$ sufficiently close to 0, we integrate over the terms involving $U$ in~\cref{eq:dens-UV} by
\begin{align*}
	\int_{U\succ 0}& |U|^{\frac{N-k-p-1}2} \zFz^{(k)}(\Sigma_1^{-2}, -\tfrac12 UV)
		\,\zFz^{(N-k)}(\Sigma_2^{-2}, -\tfrac12 U(I-V))\,(\dd U) \\
	&= 2^{(k+p)N/2} \sum_\kappa \frac{C_\kappa(\Sigma_1^{-2})}{|\kappa|! C_\kappa(I_k)} \int_{U\succ0} \zFz^{(N-k)}(\Sigma_2^{-2},-(I-V)^{1/2}U(I-V)^{1/2}) \\
	&\hspace{7cm}\cdot\ |U|^{\frac{N-k-p-1}2}\,C_\kappa(-UV)\,(\dd U) \\
	&= 2^{(k+p)N/2} \sum_\kappa \frac{C_\kappa(\Sigma_1^{-2})}{|\kappa|! C_\kappa(I_k)}
		\int_{\St_{N-k,k+p}} \int_{U\succ 0} \\
	&\qquad\cdot \etr(-U(I-V)^{1/2}H_1^\top \Sigma_2^{-2}H_1(I-V)^{1/2})|U|^{\frac{N-k-p-1}2}
		C_\kappa(-UV)\,(\dd U)\,(\dd H_1) \\
%	&= 2^{mN/2} \Gamma_m(\tfrac r2) \sum_\kappa 
%		\frac{(\frac r2)_\kappa C_\kappa(\Sigma_1^{-2})}{|\kappa|! C_\kappa(I_k)}
%		\int_{\St_{r-k,m}} |(I-V)H_1^\top \Sigma_2^{-2}H_1|^{-\frac r2}
%		C_\kappa(-\tfrac{V}{I-V}(H_1^\top \Sigma_2^{-2}H_1)^{-1}) (\dd H_1) \\
	&= 2^{(k+p)N/2} \Gamma_{k+p}(\tfrac N2) |I-V|^{-\frac N2} \int_{\St_{N-k,k+p}} 
		|H_1^\top \Sigma_2^{-2}H_1|^{-\frac N2} \\
	&\hspace{5cm}\cdot\oFz^{(k)}(\tfrac N2; \Sigma_1^{-2}, 
		-\zeta(V)(H_1^\top \Sigma_2^{-2}H_1)^{-1})\,(\dd H_1)
\end{align*}
where we perform a change-of-variables $U\to\tfrac12U$ and write the $\zFz^{(k)}$ term as a series for the first equality; apply~\cite[Cor.~1]{Shimizu2021} and the identity $\zFz(X)=\etr(X)$ for the second equality; apply~\cite[Thm.~7.2.7]{Muirhead1982} for the third equality; and introduce the shorthand $\zeta(V):=(I-V)^{-1/2}V(I-V)^{-1/2}$. Interchange of sums and integrals is facilitated by Fubini's theorem due to absolute convergence of the hypergeometric series for $\|V\|\ll1$~\cite[Thm.~4.1]{Gross1989}.\footnote{\,The computation in the previous display is a generalized Laplace transform, where $\zFz^{(k)}$ plays the role of $\etr$. A slightly different version appears in~\cite[Eq.~(6.1)]{Baker1997}.} Thus the density of $V$ near 0 can be written as
\begin{multline*}
	\phi(V) = \frac{\pi^{-kp/2} \Gamma_{k+p}(\frac N2)}{\Gamma_{k+p}(\frac{N-k}{2})
		\Gamma_k(\frac k2) |\Sigma|^{k+p}} |V|_+^{-\frac{p+1}2} |I-V|^{-\frac{2k+p+1}2} \\
	\quad\cdot\,\int_{\St_{N-k,k+p}} |H_1^\top \Sigma_2^{-2}H_1|^{-\frac N2} \oFz^{(k)}(\tfrac N2; 
		\Sigma_1^{-2}, -\zeta(V)(H_1^\top \Sigma_2^{-2}H_1)^{-1})\,(\dd H_1).
\end{multline*}
Recall that $V=(X+Y)^{-1/2}X(X+Y)^{-1/2}$, where $X=\Omega_1^\top\Sigma_1^2\Omega_1$ and $Y=\Omega_2^\top\Sigma_2^2\Omega_2$; in other words, we have derived the density function of $V=C^\top C$ near 0.

{\bf Step 2: Extract the non-zero eigenvalues of $C^\top C$.} Since the squared cosine of the largest principal angle of interest is the smallest non-zero eigenvalue of $C^\top C$, our strategy is to introduce a smaller matrix $W$ with the same non-zero eigenvalues as $V=C^\top C$. This step allows us to leverage the tools that exist for analyzing extremal eigenvalues of random matrices as opposed to specific eigenvalues in the bulk.

 The economy-sized eigenvalue decomposition $V=Q_1LQ_1^\top $, where $Q_1\in\St_{k+p,k}$ and $L$ is a $k\times k$ diagonal matrix with positive decreasing diagonal entries, has Jacobian $(\dd V)=\pi^{k(k+p)/2}\Gamma_k(\tfrac{k+p}2)^{-1}|L|^p\Delta(L)\,(\dd L)\,(\dd Q_1)$, where $\Delta(L)$ is the Vandermonde determinant of $L$~\cite[Thm.~2]{Uhlig1994}.\footnote{\,As in~\cref{fn:norm}, our Jacobian picks up an additional constant factor compared to~\cite[Thm.~2]{Uhlig1994} due to the difference in normalization of the Haar measure on $\St_{k+p,k}$. The factor is the volume of Stiefel manifold under their normalization and is computed in~\cite[Thm.~2.1.15]{Muirhead1982}.} Conversely, the full eigenvalue decomposition $Q^\top LQ=W$ with $Q\in\Or_k$ has Jacobian $(\dd W)=\pi^{k^2/2}\Gamma_k(\tfrac k2)^{-1}\Delta(L)\,(\dd L)\,(\dd Q)$, according to~\cite[Thm.~3.2.17]{Muirhead1982}. Therefore, $W\in\cS_k^{++}$ is full rank but with the same non-zero eigenvalues as $V$, and for $\|W\|\ll1$ it has density
\begin{equation}\label{eq:dens-W}
\begin{split}
	\phi(W) = \frac{\Gamma_{k+p}(\frac N2) |W|^{\frac{p-1}2}|I-W|^{-\frac{2k+p+1}2}}{\Gamma_k(\frac{k+p}2)\Gamma_{k+p}(\frac{N-k}2)|\Sigma|^{k+p}}
		 \int_{\St_{k+p,k}} \int_{\Or_k} \int_{\St_{N-k,k+p}} 
		|H^\top _1\Sigma_2^{-2}H_1|^{-\frac N2} \\
	\cdot\, \oFz^{(k)}(\tfrac N2; \Sigma_1^{-2}, -\zeta(W)
		QQ^\top _1(H^\top _1\Sigma_2^{-2}H_1)^{-1}Q_1Q^\top )\,(\dd H_1)\,(\dd Q)\,(\dd Q_1).
\end{split}
\end{equation}

{\bf Step 3: Derive the CDF $\P(K\prec W\prec I)$ in a form valid for all $0\prec K\prec I$.} Several delicate moves must be made simultaneously in this step to ensure all of our integrals are well-defined. First, we derive an alternative representation for the density of $W$, by means of analytic continuation, which is valid when its eigenvalues are near 1. We then derive the CDF $\P(K\prec W\prec I)$ via integration for SPD matrices $K$ with eigenvalues near 1. Finally, we extend the CDF formula, again by analytic continuation, so as to be valid for all $0\prec K\prec I$ via the so-called Pfaff transformation.

Observe that the joint density~\cref{eq:dens-UV} of $U,V$, as a function of two possibly complex matrices, is analytic for $U\in\cS_{k+p}^{++}$ and $V\in\cS_{k+p}^+(k)$ since the determinant, pseudodeterminant, and $\zFz$ terms are each analytic,\footnote{\,The determinant and pseudodeterminant are analytic over spaces of matrices with a given fixed rank and distinct eigenvalues; this follows from the analytic dependence of the eigenvalues away from regions of matrices with non-simple eigenvalues~\cite[Ch.~2.1]{Kato1980}. As stated previously, since such regions comprise a set of measure zero, they do not meaningfully affect our analysis. Meanwhile, the analyticity of $\zFz$ follows from the convergence analysis of its series representation~\cite[Thm.~4.1]{Gross1989}.} so the density of $W$ is also analytic for $W\in\cS_k^{++}$. Indeed, this follows from an application of Fubini's and Morera's theorems~\cite[Thm.~2.5.1]{Stein2003}, as the $L^1$ norm of $U\mapsto\phi(U,V)$ over $U\succ0$ is uniformly bounded locally over $V$ on compact sets due to the absolute convergence and continuity of the $\zFz$ hypergeometric function~\cite[Thm.~4.1]{Gross1989}.

Whenever $W\succ0$ and $\|W\|$ is sufficiently small, the hypergeometric function in~\cref{eq:dens-W} is equal by~\cite[Cor.~7.3.5]{Muirhead1982} to the function
\begin{equation}\label{eq:s3-0}
\int_{\Or_k} |I+\zeta(W)^{1/2}T^{1/2}R^\top \Sigma_1^{-2}RT^{1/2}\zeta(W)^{1/2}|^{-\frac N2}\,(\dd R),
\end{equation}
where $T:=QQ_1^\top (H_1^\top \Sigma_2^{-2}H_1)^{-1}Q_1Q^\top\succ0$. The matrix in the determinant can be rewritten as
\[
\zeta(W)^{1/2}T^{1/2}R^\top\Sigma_1^{-1} (I + \Sigma_1RT^{-1/2}\zeta(W)^{-1}T^{-1/2}R^\top\Sigma_1)\Sigma_1^{-1}RT^{1/2}\zeta(W)^{1/2},
\]
hence the determinant in~\cref{eq:s3-0} is equal to
\[
|\Sigma_1^{-1}RT^{1/2}\zeta(W)T^{1/2}R^\top\Sigma_1^{-1}||I + \Sigma_1RT^{-1/2}\zeta(W)^{-1}T^{-1/2}R^\top\Sigma_1|.
\]
The first factor in the above display can be further broken into
\[
|\Sigma_1^{-1}RT^{1/2}\zeta(W)T^{1/2}R^\top\Sigma_1^{-1}| = |\Sigma_1|^{-2}|T||\zeta(W)|,
\]
and we thus conclude that the integral~\cref{eq:s3-0} is equal to
\begin{equation}\label{eq:dens-W-1}
	\int_{\Or_k} |\zeta(W)|^{-\frac N2} |\Sigma_1|^N|T|^{-\frac N2} |I + \zeta(W)^{-1}T^{-1/2}R^\top\Sigma_1^2RT^{-1/2}|^{-\frac N2}\,(\dd R),
\end{equation}
which is certainly analytic for all $0\prec W\prec I_k$. By replacing the hypergeometric function in~\cref{eq:dens-W} with the right-hand side of~\cref{eq:dens-W-1}, we obtain a new analytic function that agrees with the right-hand side of~\cref{eq:dens-W} on an open domain in the space of complex SPD matrices. Since $\phi(W)$ is analytic for $\|W\|<1$, then by analytic continuation~\cite[Thm.~2.4.8]{Stein2003}, $\phi(W)$ is equal to the analytic function on $0\prec W\prec I_k$ given by replacing the hypergeometric function of~\cref{eq:dens-W} with the right-hand side of~\cref{eq:dens-W-1}. 

Now note that
\[
\zeta(W)^{-1} = (I-W)^{1/2}W^{-1/2}\Big((I-W)^{1/2}W^{-1/2}\Big)^\top = W^{-1}-I
\]
since $(I-W)^{1/2}W^{-1/2}=(W^{-1}-I)^{1/2}$. Repeating the same argument from above of converting $\oFz$ to an integral of a determinant but in reverse, again using analytic continuation, we obtain
\begin{multline*}
	\phi(W) = \frac{\Gamma_{k+p}(\frac N2)|\Sigma_1|^{N-k-p}}{\Gamma_k(\frac{k+p}2)
		\Gamma_{k+p}(\frac{N-k}2)|\Sigma_2|^{k+p}} |W|^{-\frac{N-p+1}2} |I-W|^{\frac{N-2k-p-1}2} \\
		\cdot\,\int_{\St_{N-k,k+p}} \int_{\St_{k+p,k}} \int_{\Or_k} \left(\frac{|(H_1^\top \Sigma_2^{-2}H_1)^{-1}|}{|T|}\right)^{\frac N2} \oFz^{(k)}\left(\tfrac N2; \Sigma_1^2, T^{-1/2}(I-W^{-1})T^{-1/2}\right) \\
	\qquad\qquad\cdot(\dd Q)\,(\dd Q_1)\,(\dd H_1),
\end{multline*}
whenever $0\prec W\prec I$ and $\|I-W^{-1}\|$ is sufficiently small. Thus, for $0\prec K\prec I_k$ and $\|I-K^{-1}\|$ sufficiently small, we have
\begin{align}\label{eq:dist-W}
	&\P(K\prec W\prec I)  \\
	&= \frac{\Gamma_{k+p}(\frac N2)|\Sigma_1|^{N-k-p}}{\Gamma_k(\frac{k+p}2)
		\Gamma_{k+p}(\frac{N-k}2)|\Sigma_2|^{k+p}} \int_{\St_{N-k,k+p}} \int_{\St_{k+p,k}}
		\left(\frac{|(H_1^\top \Sigma_2^{-2}H_1)^{-1}|}{|Q_1^\top (H_1^\top \Sigma_2^{-2}H_1)^{-1}Q_1|}
		\right)^{\frac N2} \nonumber \\
	&\quad\cdot\,\sum_\kappa \frac{(\frac N2)_\kappa C_\kappa(\Sigma_1^2)}
		{|\kappa|! C_\kappa(I_k)} \int_{\Or_k} \int_{K\prec W\prec I} |W|^{-\frac{N-p+1}2} |I-W|^{\frac{N-2k-p-1}2} \nonumber \\
	&\quad\cdot\,C_\kappa\left((I-W^{-1})T^{-1}\right)\,(\dd W)\,(\dd Q)\,(\dd Q_1)\,(\dd H_1). \nonumber
\end{align}

For the inner integral, we transform $Z=W^{-1}-I$ which has Jacobian $(\dd W)=|I+Z|^{-(k+1)}\,(\dd Z)$ as given in~\cite[Eq.~(1.3.13)]{Gupta2000}, yielding
\begin{multline*}
\int_{K\prec W\prec I} |W|^{-\frac{N-p+1}2}|I-W|^{\frac{N-2k-p-1}2} C_\kappa\left((I-W^{-1})T^{-1}\right)\,(\dd W) \\
= \int_{0\prec Z\prec K^{-1}-I} |Z|^{\frac{N-2k-p-1}2} C_\kappa(-ZT^{-1})\,(\dd Z).
\end{multline*}
We transform again $U=T^{-1/2}ZT^{-1/2}$ with Jacobian $(\dd U)=|T|^{-\frac{k+1}2}(\dd Z)$ \cite[Eq.~(1.3.5)]{Gupta2000} to obtain the integral
\[
|T|^{\frac{N-k-p}2}\int_{0\prec U\prec T^{-1/2}(K^{-1}-1)T^{-1/2}}|U|^{\frac{N-2k-p-1}2} C_\kappa(-U)\,(\dd U).
\]
Applying~\cref{lem:partial-integral} finally yields
\[
\frac{(\frac{N-k-p}2)_\kappa}{(\frac{N-p+1}2)_\kappa} \Beta_k(\tfrac{N-k-p}2,\tfrac{k+1}2) |K^{-1}-I|^{\frac{N-k-p}2} C_\kappa\left((I-K^{-1}) T^{-1}\right),
\]
recalling that $\Beta_k(a,b)$ is the multivariate beta function. Substituting into~\cref{eq:dist-W} and using the identity $\Gamma_{k+p}(\tfrac N2)\Gamma_k(\tfrac{N-k-p}2) = \Gamma_k(\tfrac N2)\Gamma_{k+p}(\tfrac{N-k}2)$, we obtain
\begin{multline}\label{eq:pre-pfaff}
	\P(K\prec W\prec I) = \frac{\Gamma_k(\frac N2)\Gamma_k(\frac{k+1}2)|\Sigma_1|^{N-k-p}}
		{\Gamma_k(\frac{k+p}2)\Gamma_k(\frac{N-p+1}2)|\Sigma_2|^{k+p}} 
		|K^{-1}-I|^{\frac{N-k-p}2} \\
	\cdot\,\int_{\St_{N-k,k+p}} \int_{\St_{k+p,k}} \int_{\Or_k} \left(\frac{|(H_1^\top \Sigma_2^{-2}H_1)^{-1}|}
		{|Q_1^\top (H_1^\top \Sigma_2^{-2}H_1)^{-1}Q_1|}\right)^{\frac N2} \sum_\kappa 
		\frac{(\tfrac N2)_\kappa (\tfrac{N-k-p}2)_\kappa C_\kappa(\Sigma_1^2)}
		{|\kappa|!(\tfrac{N-p+1}2)_\kappa C_\kappa(I_k)} \\
	\cdot\,C_\kappa\left((I-K^{-1})T^{-1}\right) \,(\dd Q)\,(\dd Q_1)\,(\dd H_1).
\end{multline}
The sum over $\kappa$ can be rewritten as $\int_{\Or_k}\tFo(\tfrac N2,\tfrac{N-k-p}2;\tfrac{N-p+1}2;-E)\,(\dd R)$ where
\[
E=T^{-1/2}K^{-1/2}(I-K)^{1/2}R^\top\Sigma_1^2R(I-K)^{1/2}K^{-1/2}T^{-1/2}\succ0.
\]
Note that $E^{1/2}(I+E)^{-1}E^{1/2}\succ0$ has the same eigenvalues as
\[
F=\Sigma_1R(I-K)^{1/2}\left[K^{1/2}TK^{1/2}+(I-K)^{1/2}R^\top\Sigma_1^2R(I-K)^{1/2}\right]^{-1}(I-K)^{1/2}R^\top\Sigma_1,
\]
which is analytic in $K$. Now the Pfaff transformation~\cite[Thm.~7.4.3]{Muirhead1982} yields
\begin{multline}\label{eq:K}
\P(K\prec W\prec I) = \frac{\Gamma_k(\frac N2)\Gamma_k(\frac{k+1}2)}
	{\Gamma_k(\frac{k+p}2)\Gamma_k(\frac{N-p+1}2)} \\
\cdot\int_{\St_{N-k,k+p}} \int_{\St_{k+p,k}} \int_{\Or_k} \int_{\Or_k} \frac{|H_1^\top \Sigma_2^{-2}H_1|^{-\frac N2}}{|\Sigma_2|^{k+p}
	|Q_1^\top (H_1^\top \Sigma_2^{-2}H_1)^{-1}Q_1|^{\frac{k+p}2}} \\
\cdot\,\tFo\left(\tfrac{-p+1}2, \tfrac{N-k-p}2; \tfrac{N-p+1}2; F\right) |F|^{\frac{N-k-p}2}\,(\dd R)\,(\dd Q)\,(\dd Q_1)\,(\dd H_1),
\end{multline}
again for $\|I-K^{-1}\|$ sufficiently small. However, the map $K\mapsto\P(K\prec W\prec I)$ is analytic for all $0\prec K\prec I$ due to the analyticity of $\phi(W)$, as is the right-hand side of~\cref{eq:K}; by a final application of the analytic continuation argument,~\cref{eq:K} is valid for all $0\prec K\prec I$.

{\bf Step 4: Simplify into a bound on the largest principal angle.} Fix an angle $\theta$ and set $K=\cos^2(\theta)I$, so that
\[
F = \sin^2(\theta)\Sigma_1\left[\cos^2(\theta)R\,TR^\top + \sin^2(\theta)\Sigma_1^2\right]^{-1}\Sigma_1.
\]
Recall that $T=QQ_1^\top(H_1^\top\Sigma_2^{-2}H_1)^{-1}Q_1Q^\top$. By a change of variables $Q_1\mapsto Q_1Q^\top R^\top$, we may eliminate the dependence of the integrand in~\cref{eq:K} on $Q$ and $R$. Secondly, observe that the quantity
\[
	\frac{|H_1^\top \Sigma_2^{-2}H_1|^{-N/2}}{|\Sigma_2|^{k+p}|Q_1^\top (H_1^\top \Sigma_2^{-2}H_1)^{-1}Q_1|^{(k+p)/2}}
\]
is the joint density function of $H_1\sim\MACG(\Sigma_2^2)$, $Q_1\sim\MACG(H_1^\top \Sigma_2^{-2}H_1)$, the latter in the sense of a regular conditional distribution depending on $H_1$, where $\MACG$ denotes the matrix angular central Gaussian distribution on their respective Stiefel manifolds~\cite[Thm.~2.4.2]{Chikuse2003}. These random matrices can be realized as orthonormalized Gaussians as follows. Let $X\sim\cN_{N-k,k+p}(0,1)$ and $Y\sim\cN_{k+p,k}(0,1)$ be independent random matrices. Then $H_1=\Sigma_2X(X^\top \Sigma_2^2X)^{-1/2}$ has the $\MACG(\Sigma_2^2)$ distribution. Set
\[
	S = H_1^\top \Sigma_2^{-2}H_1 = (X^\top \Sigma_2^2X)^{-1/2}X^\top X(X^\top \Sigma_2^2X)^{-1/2}
\]
and recall that non-zero eigenvalues are preserved when matrix products commute, hence $S$ has the same eigenvalues as $X^\top X(X^\top \Sigma_2^2X)^{-1}$. The QR decomposition $X=\widetilde H_1R$, where $\widetilde H_1\sim\Unif(\St_{N-k,k+p})$ by orthogonal invariance of $X$, further implies that $S$ has the same eigenvalues as $(\widetilde H_1^\top \Sigma_2^2\widetilde H_1)^{-1}$ almost surely, as $X$ is almost surely full rank. Applying a similar argument to $Q_1=S^{1/2}Y(Y^\top SY)^{-1/2}$ shows that $\lambda^\downarrow(Q_1^\top (H_1^\top \Sigma_2^{-2}H_1)^{-1}Q_1)$ and $\lambda^\downarrow((\widetilde Q_1^\top (\widetilde H_1^\top \Sigma_2^2\widetilde H_1)^{-1}\widetilde Q_1)^{-1})$ are equal in distribution, where $\widetilde Q_1\sim\Unif(\St_{k+p,k})$ is selected independently of $\widetilde H_1$, and $\lambda^\downarrow$ denotes the tuple of eigenvalues of a matrix set in decreasing order. 

These simplifications altogether yield
\begin{multline}\label{eq:tilde-cdf}
	\P(W\succ \cos^2(\theta)I) = \frac{\Gamma_k(\frac N2)\Gamma_k(\frac{k+1}2)}
		{\Gamma_k(\frac{k+p}2)\Gamma_k(\frac{N-p+1}2)} \int_{\St_{N-k,k+p}} \int_{\St_{k+p,k}} |\widetilde S_\theta|^{\frac{N-k-p}2} \\
	\,\cdot\,\tFo(\tfrac{-p+1}2,\tfrac{N-k-p}2;\tfrac{N-p+1}2;\widetilde S_\theta)
		\,(\dd Q_1)\,(\dd H_1)
\end{multline}
where 
\begin{equation}\label{eq:tilde-S}
\widetilde S_\theta = (I+\cot^2(\theta)\Sigma_1^{-1}(Q_1^\top(H_1^\top\Sigma_2^2H_1)^{-1} Q_1)^{-1}\Sigma_1^{-1})^{-1}.
\end{equation}
Notice that $\widetilde S_\theta$ is equal to the $S_\theta$ of~\cref{eq:S} since $H_1^\top\Sigma_2^2H_1$ and $Q_1^\top(H_1^\top\Sigma_2^2H_1)^{-1}Q_1$ are almost surely invertible, hence the $P_\perp$ factor in $S_\theta$ is the identity. To conclude, recall that the eigenvalues of $W$ are the squared cosines of the principal angles between $U_1$ and $\widetilde U_1$, so the bound $W\succ \cos^2(\theta)I$ is equivalent to the bound on the largest principal angle $\theta_1(U_1,\widetilde U_1)<\theta$. This proves~\cref{eq:cdf} in the full-rank regime $r=N\ge2k+p$.

{\bf Part II: Rank-deficient regime.} In the low-rank regime $k+p<r<2k+p$, we proceed by a continuity argument. We further split the proof into the cases $N\ge 2k+p$ and $N<2k+p$.

{\bf Case 1: $N\ge 2k+p$.} Recall that $\cos\theta_1$ is the smallest singular value of $C=I_{N,k}^\top\Sigma\Omega(\Omega^\top\Sigma^2\Omega)^{-1/2}$, which is continuous as a function of the diagonal of $\Sigma$ as long as $\Omega^\top\Sigma^2\Omega$ is invertible. This holds almost surely when $\rank(A)>k+p$ (the case $\rank(A)\le k+p$ is trivial as the RRF returns the entire range of $A$ exactly). It is also straightforward to verify that the distribution of $\theta_1$ has no point masses whenever $\rank(A)>k+p$ by Gaussianity of $\Omega$. Therefore, the CDF of $\theta_1$ varies continuously with the diagonal of $\Sigma$, so it suffices to show that~\cref{eq:tilde-cdf} converges to~\cref{eq:cdf} when some diagonal entries of $\Sigma_2$ tend to 0. In particular, observe that the eigenvalues of $\widetilde S_\theta$ are uniformly bounded within $[0,1]$ over all $\Sigma_2\succ 0$, and $S_\theta\mapsto |S_\theta|^{\frac{N-k-p}2}\tFo(\tfrac{-p+1}2,\tfrac{N-k-p}2;\tfrac{N-p+1}2;S_\theta)$ is continuous in the eigenvalues of $S_\theta$ and therefore also uniformly bounded over all $\Sigma_2\succ 0$. The bounded convergence theorem, in conjunction with compactness of Stiefel manifolds, now implies that we only need to show that $\widetilde S_\theta$ of~\cref{eq:tilde-S} converges to $S_\theta$ of~\cref{eq:S} in operator norm as some diagonal entries of $\Sigma$ tend to 0. In particular, it suffices to show that $(Q_1^\top(H_1^\top\Sigma_2^2H_1)^{-1}Q_1)^{-1}$ converges to $(P_\perp Q_1^\top(H_1^\top\Sigma_2^2H_1)^\dagger Q_1P_\perp)^\dagger$ by continuity of matrix inversion. Finally, observe that $r\ge 2k+p$ implies $\rank(\Sigma_2^2) \ge k+p$ such that~\cref{eq:tilde-S} is well-defined and equal to~\cref{eq:S} almost surely. Thus, the remaining case to consider carefully is when exactly $p<s<k+p$ diagonal entries of $\Sigma_2$ are non-zero.

Fix the largest $k+s$ singular values of $A$ to be non-zero; we send the remaining $N-k-s$ singular values to 0. In other words, let $\Sigma_1$ be fixed and let $\Sigma_2^\epsilon=(\sigma_{k+1},\dots,\sigma_{k+s},\epsilon,\dots,\epsilon)$ with $\sigma_{k+1}\ge\dots\ge\sigma_{k+s}>\epsilon>0$; as a limiting case, we set $\Sigma_2^0 = (\sigma_{k+1},\dots,\sigma_{k+s},0,\dots,0)$. Fix $H_1\in\St_{N-k,k+p}$ from the set of full measure in which $H_1^\top(\Sigma_2^0)^2H_1$ has rank $s$, and define $S_\epsilon := H_1^\top(\Sigma_2^\epsilon)^2H_1$ and $S_0:=H_1^\top(\Sigma_2^0)^2H_1$. By Cauchy's interlacing theorem, the smallest $k+p-s$ eigenvalues of $S_\epsilon$ must all equal $\epsilon^2$. On the other hand, if $\alpha_1\ge\dots\ge\alpha_s>0$ are the non-zero eigenvalues of $S_0$, then the largest $s$ eigenvalues of $S_\epsilon$ are $\beta_j = \alpha_j+O(\epsilon^2)$ for $j=1,\dots,s$, by Weyl's inequalities~\cite[Thm.~III.2.1]{Bhatia1997}.

 Let $S_\epsilon$ have the eigendecomposition
 \[
 	S_\epsilon = \begin{bmatrix}R_1 & R_2\end{bmatrix}
 	\begin{bmatrix} B & \\ & \epsilon^2 I_{k+p-s}\end{bmatrix}
	\begin{bmatrix} R_1^\top \\ R_2^\top\end{bmatrix},
\]
where $B$ is an $s\times s$ diagonal matrix of $\beta_1\ge\dots\ge\beta_s\ge\epsilon^2$, $R=[R_1\ R_2]\in\Or_{k+p}$, and $R_1\in\St_{k+p,s}$, $R_2\in\St_{k+p,k+p-s}$. Let 
\[
	\sE_\epsilon = \{Q_1\in\St_{k+p,k}:\|\sin\Theta(Q_1,R_2)\|^2\le 1-\epsilon\}
\]
be the set of orthonormal $k$-frames that are bounded away from every subspace of $R_1$. 
%By~\cite[Thm.~2.4 and Eq.~(2.10)]{Knyazev2010}, we have
%\[
%	\max_{j=1,\dots,k+p-s}|\epsilon^{-2} - \lambda_j^\downarrow(Q_1^\top S_\epsilon^{-1}Q_1)| \le \epsilon^{-2} \|\sin\Theta(Q_1,R_2)\|^2 \le \epsilon^{-2}(1-\epsilon)
%\]
%whenever $Q_1\in\sE_\epsilon$, hence $\lambda_j^\downarrow(Q_1^\top S_\epsilon^{-1}Q_1) \ge \epsilon^{-1}$ for $j=1,\dots,k+p-s$. 
For any $Q_1\in\sE_\epsilon$, set $Q_{11}=R_1^\top Q_1\in\R^{s\times k}$ and $Q_{12}=R_2^\top Q_1\in\R^{(k+p-s)\times k}$, so that
\begin{equation}\label{eq:cdf-cty}
	Q_1^\top S_\epsilon^{-1}Q_1 = Q_{11}^\top B^{-1}Q_{11} + \epsilon^{-2}Q_{12}^\top Q_{12}.
\end{equation}
Since $\sigma(Q_{12})=\cos\Theta(Q_1,R_2)$,  then $\sigma_\mathrm{min}^2(Q_{12})\ge\epsilon$ by definition of $\sE_\epsilon$. Therefore, $\epsilon^{-2}Q_{12}^\top Q_{12}$ is of rank $k+p-s$ and has smallest non-zero eigenvalue $\ge\epsilon^{-1}$. Additionally, Weyl's inequalities imply that the smallest $s-p$ eigenvalues of $Q_1^\top S_\epsilon^{-1}Q_1$ are uniformly bounded as $\epsilon\to0$, since 
\begin{equation}\label{eq:cdf-cty-1}
\lambda^\downarrow_{k+p-s+j}(Q_1^\top S_\epsilon^{-1}Q_1) \le \lambda_j^\downarrow(Q_{11}^\top B^{-1}Q_{11})
\le \beta_{s-j+1}^{-1} = \alpha_{s-j+1}^{-1} + O(\epsilon^2)
\end{equation}
for $j=1,\dots,s-p$; they are also bounded from below by $\beta_1^{-1}>0$ by Cauchy's interlacing theorem.

Let $K=[K_1\ K_2]\in\Or_k$ be such that $K_1\in\St_{k,s-p}$ spans $\mathrm{null}(Q_{12})$ and $K_2\in\St_{k,k+p-s}$ spans its orthogonal complement. We write $Q_1^\top S_\epsilon^{-1}Q_1$ as a $2\times 2$ block matrix in the $K$ basis,
\begin{multline*}
K^\top Q_1^\top S_\epsilon^{-1}Q_1K \\
= \begin{bmatrix}
K_1^\top Q_{11}^\top B^{-1}Q_{11}K_1 & K_1^\top Q_{11}^\top B^{-1}Q_{11}K_2 \\
K_2^\top Q_{11}^\top B^{-1}Q_{11}K_1 & K_2^\top(Q_{11}^\top B^{-1}Q_{11}+\epsilon^{-2}Q_{12}^\top Q_{12})K_2
\end{bmatrix}
=: \begin{bmatrix} D_1 & D_2 \\ D_2^\top & D_4 \end{bmatrix}.
\end{multline*}
We send $\epsilon\to0$ along with any sequence of $Q_1\in\sE_\epsilon$. The eigenvalues of $D_1$ remain uniformly bounded away from 0 and $+\infty$ by~\cref{eq:cdf-cty-1}, and the singular values of $D_2$ are also uniformly bounded from above. Every eigenvalue of $D_4$ tends to $+\infty$, since $D_4\in\R^{(k+p-s)\times(k+p-s)}$ while $\rank(Q_{12}^\top Q_{12})=k+p-s$. Both $D_1,D_4$ are invertible for $\epsilon>0$. Inversion of block matrices~\cite[Sec.~0.7.3]{Horn1985} then yields
\[
K^\top(Q_1^\top S_\epsilon^{-1}Q_1)^{-1}K = \begin{bmatrix}
D_{1\cdot 4}^{-1} & -D_{1\cdot 4}^{-1}D_2D_4^{-1} \\
-D_4^{-1}D_2^\top D_{1\cdot 4}^{-1} & D_{4\cdot 1}^{-1}
\end{bmatrix}
\]
where $D_{1\cdot4}=D_1-D_2D_4^{-1}D_2^\top$ and $D_{4\cdot1}=D_4-D_2^\top D_1^{-1}D_2$ are Schur complements. As $\epsilon\to0$, $D_{1\cdot4}^{-1}\to D_1^{-1}$ and $D_{4\cdot1}^{-1}\to0$ in operator norm, hence $(Q_1^\top S_\epsilon^{-1}Q_1)^{-1}\to (K_1K_1^\top Q_{11}^\top B^{-1}Q_{11} K_1K_1^\top)^\dagger$. But $K_1K_1^\top$ is precisely the projection onto the orthogonal complement of $Q_1^\top R_2$, thus $K_1K_1^\top\to P_\perp$ as defined in~\Cref{thm:cdf}. Similarly, $Q_{11}^\top B^{-1}Q_{11} \to Q_1^\top S_0^\dagger Q_1$ by construction. We conclude that $(Q_1^\top S_\epsilon^{-1}Q_1)^{-1}\to (P_\perp Q_1^\top S_0^\dagger Q_1P_\perp)^\dagger$ as $\epsilon\to0$, in operator norm.

In summary, we have shown that for almost every $H_1\in\St_{N-k,k+p}$ and $\delta>0$, there is sufficiently small $\epsilon>0$ such that
\[
\P\left(Q_1\in\St_{k+p,k} : \left\|(Q_1^\top S_\epsilon^{-1}Q_1)^{-1}
- (P_\perp Q_1^\top S_0^\dagger Q_1P_\perp)^\dagger\right\| < \delta \right) \ge \P(\sE_\epsilon) \to 1
\]
as $\epsilon\to0$, with respect to the normalized Haar measure on $\St_{k+p,k}$. In other words, the two matrices in the above display converge in probability as $\epsilon\to0$, so the convergence holds almost surely in $\St_{k+p,k}$ by passing to an appropriate subsequence of $\epsilon$~\cite[Thm.~2.3.2]{Durrett2019}. This constitutes the desired convergence.

{\bf Case 2: $N<2k+p$.} Given partition~\cref{eq:A-partition}, consider the $(2k+p)\times(2k+p)$ augmented matrix
\[
\widehat\Sigma = \begin{bmatrix}
\Sigma_1 & 0 \\ 0 & \widehat\Sigma_2
\end{bmatrix} \in \R^{(2k+p)\times(2k+p)}, \qquad
\widehat\Sigma_2 = \begin{bmatrix}
\Sigma_2 & 0 \\ 0 & 0
\end{bmatrix}\in\R^{(k+p)\times(k+p)}.
\] 
Since marginals of random matrices with i.i.d.~standard Gaussian entries are also standard Gaussian, it is clear that the distribution of the largest principal angle $\theta_1$ does not change whether sketching $\Sigma$ with $\Omega\sim\cN_{N,k+p}(0,1)$ versus sketching $\widehat\Sigma$ with $\widehat\Omega\sim\cN_{2k+p,k+p}(0,1)$. Therefore, all of our previous arguments hold once $N$ is replaced everywhere by $2k+p$. This completes the proof.
\end{proof}

\subsection{Terminating series for odd $p$}
Remarkably, the hypergeometric series in~\cref{eq:cdf} has a particularly simple expression when $p$ is odd. In this case, the hypergeometric series terminates, which further allows us to employ a reflection formula that eliminates the leading constant factor $K_{N,k,p}$. The representation~\cref{eq:cdf-reflected}, which we prove below, is important numerically as its integrand can be computed exactly, and we exploit this in our numerical experiments (see~\Cref{ss:numerics}). 

We begin with a minor extension to the reflection formula of \cite{Richards2024}.

\begin{lemma}\label{lem:reflection}
Let $X\in\cS_k^{++}$. Suppose $a\le0$ is an integer and $b,c\in\C$ satisfy $\tfrac{j+1}{2}-c\notin\N$ for all $j=1,\dots,k$, and $c-b = \tfrac{k+1}{2}$. Then
\[
\tFo(a,b;c;X) = \frac{\Gamma_k(c)\Gamma_k(c-a-b)}{\Gamma_k(c-a)\Gamma_k(c-b)}
	\sum_{\ell=0}^{k|a|}\sum_{\substack{\kappa\vdash\,\ell\\\kappa_1\le|a|}} \frac{(b)_\kappa}{\ell!} C_\kappa(I-X),
\]
where the inner summation is over all partitions $\kappa=(\kappa_1,\kappa_2,\dots)$ of $\ell$ for which the largest component $\kappa_1$ is bounded by $|a|$.
\end{lemma}
\begin{proof}
The conditions on $a$ and $c$ ensure that $\tFo(a,b;c;X)$ exists and that both sides are polynomials in $X$, since $a\le0$ is an integer. There exists $h\in\C$ sufficiently small so that $\tfrac{j+1}{2}-(c+h)\notin\N$ and $(c+h)-a-b-\tfrac{k-j}{2}\notin\N$ for all $j=1,\dots,k$. Then apply the reflection formula \cite[Prop.~3.1]{Richards2024} to obtain
\begin{multline*}
\tFo(a,b;c+h;X) = \frac{\Gamma_k(c+h)\Gamma_k(c+h-a-b)}
	{\Gamma_k(c+h-a)\Gamma_k(c+h-b)} \tFo(a,b;a-h;I-X) \\
= \frac{\Gamma_k(c+h)\Gamma_k(c+h-a-b)}{\Gamma_k(c+h-a)\Gamma_k(c+h-b)}
	\sum_{\ell=0}^{k|a|}\sum_{\substack{\kappa\vdash\,\ell\\\kappa_1\le|a|}} \frac{(a)_\kappa(b)_\kappa}{\ell!(a-h)_\kappa}
	C_\kappa(I-X)
\end{multline*}
Taking $|h|\to0$ yields the result, noting that $(a)_\kappa\neq0$ for all $\kappa$ such that $\kappa_1\le|a|$.
\end{proof}

\begin{corollary}\label{cor:terminate}
Assume the hypotheses of~\Cref{thm:cdf}. If $k+p<\rank(A)$ and $p$ is odd, then for $0\le\theta\le\frac\pi2$, we have
\begin{equation}\label{eq:cdf-reflected}
\Phi_{\Sigma,k,p}(\theta) = \E\left[|S_\theta|^{\frac{N-k-p}2} \sum_{\ell=0}^{k(p-1)/2} \sum_{\substack{\kappa\vdash\,\ell\\\kappa_1\le\frac{p-1}2}} \frac{(\frac{N-k-p}2)_\kappa}{\ell!}C_\kappa(I-S_\theta)\right],
\end{equation}
where $S_\theta$ is defined in~\cref{eq:S}.
\end{corollary}
\begin{proof}
Let $a,b,c$ be the parameters in the hypergeometric series of~\cref{eq:cdf}. Odd $p$ implies $a$ is a non-positive integer, and $N>k+p$ implies $c-\frac{j+1}2>0$ for all $j=1,\dots,k$. Clearly $c-b=\frac{k+1}2$, so~\Cref{lem:reflection} implies the result.
\end{proof}

\subsection{Heuristics}
The formula~\cref{eq:cdf} may seem intimidating, but we can successively break it down in digestible pieces. In the simplest setting, let us suppose $k=1$, $p=0$, and $\Sigma=I_N$, in which case the random matrix $S_\theta$ of~\cref{eq:S} simply becomes the scalar matrix $\sin^2(\theta)I$. On the other hand, the hypergeometric function of matrix argument becomes the the familiar Gaussian hypergeometric function of scalar argument, and~\cref{eq:cdf} simplifies as
\begin{equation}\label{eq:heuristic}
\Phi_{I_N,1,0}(\theta) = \frac{\Gamma(\tfrac N2)}{\Gamma(\tfrac12)\Gamma(\tfrac{N+1}2)}\sin^{N-1}(\theta)\,\tFo(\tfrac12,\tfrac{N-1}2;\tfrac{N+1}2;\sin^2(\theta)).
\end{equation}
Indeed, for our specific choices of $k$, $p$, and $\Sigma$, the multivariate gamma functions in~\cref{eq:cdf} become standard gamma functions; $\widetilde S_\theta$ is simply $\sin^2(\theta)I$, and the integrals disappear since there is no dependence on $Q_1,H_1$. Now, from the well-known identity~\cite[Eq.~(8.17.7)]{NIST}, one sees that~\cref{eq:heuristic} is nothing more than an incomplete beta function with parameters $\tfrac{N-1}2$ and $\tfrac12$, evaluated at $\sin^2(\theta)$ and normalized to be a valid CDF. In other words, $\Phi_{I_N,1,0}(\theta)$ is the CDF\ of the univariate $\mathrm{Beta}(\tfrac{N-1}2,\tfrac12)$ distribution scaled according to $\sin^2(\theta)$.

What does the beta distribution have to do with random subspaces? In fact, it has been known since the 1980s that the $\mathrm{Beta}(\tfrac{N-1}2,\tfrac12)$ distribution describes the squared sine of the angle between two uniformly sampled random lines through the origin in $\R^N$~\cite[Thm.~1.5.7]{Muirhead1982}, and this fact was used in~\cite{Frankl1990} to improve the classical Johnson--Lindenstrauss lemma. Later came a generalization to two uniformly sampled random $k$-dimensional subspaces of $\R^N$, upon whose proof ours is built~\cite{Absil2006}. Indeed, when $k$ is allowed to be $\ge1$ while $p=0$ and $\Sigma=I_N$ remain fixed, our result reduces exactly to~\cite[Eq.~(12)]{Absil2006}.

As for our formula~\cref{eq:cdf}, we generalize even further in two important ways. First, by permitting $p\ge0$, we allow the two subspaces in $\R^N$ to be non-equidimensional. This aspect of the generalization is what demands our extension of classical results in multivariate statistics to singular random matrices, via~\cref{lem:wishart-integral,lem:quadratic-form}. Second, by permitting $\Sigma$ to be any positive definite diagonal matrix, we allow one of the two subspaces to be sampled non-uniformly and instead be biased according to the relative weights of the diagonal entries of $\Sigma$. Notice that for general values of $\Sigma$, the $\sin^2(\theta)$ argument becomes the matrix $S_\theta$, which can be thought of as a multivariate version of an elliptically scaled squared sine function (the Jacobi sn function in the univariate case~\cite[Ch.~22]{NIST}) in which the major axes of the ellipsoid are scaled according to $\Sigma_1$ while the minor axes are scaled according to $\Sigma_2$.

We remark that we employ this point of view of~\cref{eq:cdf} as a generalized beta distribution both in the proof of~\cref{thm:cdf} above, and in the proof of our second main result,~\cref{thm:Fgap-bound}, below. 

A second, related perspective of~\cref{eq:cdf} is through the series expansion of the hypergeometric function. When $p$ is odd, we consider the terminating series of~\cref{eq:cdf-reflected}. If one takes $\mu\to\infty$, the summation simply becomes the series representation for $\oFz(\tfrac{N-k-p}2;I-S_\theta)$, which converges to $|S_\theta|^{-(N-k-p)/2}$ by~\cite[Cor.~7.3.5]{Muirhead1982}. In other words, when $p$ is odd, the CDF~\cref{eq:cdf} can be viewed as a ratio with an infinite series in the denominator, and a finite partial sum of that infinite series in the numerator; in particular, the denominator is $|S_\theta|^{-(N-k-p)/2}$ expanded as a zonal series around $I_k$, while the numerator is a partial sum of that zonal series:
\[
	\Phi_{\Sigma,k,p}(\theta) = \E\left[\frac{\sum_{\ell=0}^{\mu k}\sum_{\kappa\vdash\ell,\,\kappa_1\le\mu} \tfrac1{\ell!}\big(\tfrac{N-k-p}2\big)_\kappa C_\kappa(I-S_\theta)}{\sum_{\ell=0}^\infty\sum_{\kappa\vdash\,\ell}\tfrac1{\ell!}\big(\tfrac{N-k-p}2\big)_\kappa C_\kappa(I-S_\theta)}\right], \qquad\text{$p$ odd},\ \mu=\tfrac{p-1}2.
\]
Such a representation is significant for two reasons. First, we see that as $p$ increases, more of the partial sum appears in the numerator, hence the probability that the angular error is small also increases. Second, it suggests the existence of Pad\'e-type approximations to the CDF, akin to Pad\'e approximations for the univariate Gaussian hypergeometric function \cite{Baker1975}.

\subsection{Numerical experiments}\label{ss:numerics}

Numerical evidence confirms~\cref{thm:cdf}. As discussed above,~\cref{eq:cdf} can be written as a terminating series~\cref{eq:cdf-reflected} when $p$ is odd, conveniently lending itself to the extremely fast, exact algorithm of~\cite{Koev2006}, which exploits a recursion property to compute the zonal polynomials. On the other hand, when $p$ is even, we approximate~\cref{eq:cdf} by truncating the series expansion of the hypergeometric term. Therefore, the formula for the CDF~\cref{eq:cdf} can be evaluated very quickly despite its complexity.~\cref{fig:cdf-numerics} demonstrates excellent agreement of our formula~\cref{eq:cdf} with empirical observations.

\section{Probabilistic error bounds for the RRF}\label{s:applications}
In practice, one often does not have access to the singular values of the matrix $A$ and so cannot apply~\cref{thm:cdf} directly. However, by having an exact expression for the RRF's subspace approximation error, we observe several properties of the RRF.

\subsection{Initial observations}\label{ss:auxiliary}
We first simplify~\cref{eq:cdf} to aid in its analysis. Let
\begin{equation}\label{eq:J}
\cJ_{N,k,p}(S) := K_{N,k,p}\,|S|^{\tfrac{N-k-p}2}\tFo(\tfrac{-p+1}2,\tfrac{N-k-p}2;\tfrac{N-p+1}2;S)
\end{equation}
be the integrand of~\cref{eq:cdf}, defined for all symmetric $k\times k$ matrices $S$ such that $0\prec S\prec I$, and integer parameters $N,k,p\ge0$ such that $N\ge k+p$.

\begin{lemma}\label{lem:J-monotone}
$\cJ_{N,k,p}(S)$ is non-negative and increasing in the eigenvalues of $S$.
\end{lemma}
\begin{proof}
We first claim that the right-hand side of~\cref{eq:K} with $\Sigma=I$ is simply $\cJ_{N,k,p}(I-K)$, for any symmetric $0\prec K\prec I$. Indeed, note that when $\Sigma_1,\Sigma_2$ are both identity matrices, the first fraction in the integrand of~\cref{eq:K} reduces as
\[
\frac{|H_1^\top \Sigma_2^{-2}H_1|^{-\tfrac N2}}{|\Sigma_2|^{k+p}
		|Q_1^\top (H_1^\top \Sigma_2^{-2}H_1)^{-1}Q_1|^{\tfrac{k+p}2}} = 1.
\]
The matrix argument in the hypergeometric function is $I-K$ since the matrix in the square brackets becomes the identity, whereas the final fraction in the integrand reduces to $|I-K|$ by the same argument. Therefore, the integrand in~\cref{eq:K} becomes
\[|I-K|^{(N-k-p)/2}{}_2F_1(\tfrac{-p+1}2,\tfrac{N-k-p}2;\tfrac{N-p+1}2;I-K).\]
This is independent of $Q_1$, $H_1$, or $Q$, so the integrals integrate out to 1, as desired.

Since~\cref{eq:K} is a cumulative distribution function, it is non-negative and decreasing in $K$ with respect to the L\"owner order and hence also in the eigenvalues of $K$. Setting $S=I-K$ means $\cJ_{N,k,p}(S)$ is increasing in the eigenvalues of $S$. Finally, the determinant and the hypergeometric function depend only on the eigenvalues of their arguments, so the same holds for $\cJ_{N,k,p}(S)$.
\end{proof}

Next, we mark a few monotonicity properties of $\Phi_{\Sigma,k,p}(\theta)$ with respect to its parameters.

\begin{lemma}\label{lem:monotone}
Assume the hypotheses of~\Cref{thm:cdf}. For each $0\le\theta\le\pi/2$, 
\begin{enumerate}[leftmargin=8mm]
\item $\Phi_{\Sigma,k,p}(\theta)$ is increasing in $p$;
\item $\Phi_{\Sigma,k-j,p+j}(\theta)$ is increasing in $0\le j< k$; and
\item $\Phi_{\Sigma,k,p}(\theta)\ge\Phi_{\Sigma^-,k,p-1}(\theta)$ if $N>k$ and $p\ge1$, where $\Sigma^-$ is any $(N-1)\times(N-1)$ principal submatrix of $\Sigma$ obtained by deleting a row and column of index $>k$, i.e., from the $\Sigma_2$ block.
\end{enumerate}
\end{lemma}
\begin{proof}
Recall that, up to a change of basis, $\theta_1$ is the largest principal angle between the column spaces of $I_{N,k}$ and $\Sigma\Omega$, where $\Omega\sim\cN_{N,k+p}(0,1)$. Suppose $0\le p'<p$. Let $\Omega'$ be given by the first $k+p'$ columns of $\Omega$ and let $\theta_1'$ be the largest principal angle between $I_{N,k}$ and $\Sigma\Omega'$. The column space of $\Sigma\Omega'$ is strictly contained in the column space of $\Sigma\Omega$, so $\theta_1\le\theta_1'$ by the variational definition of principal angles~\cite[Eq.~(6.4.3)]{Golub2013}. Moreover, we have $\Omega'\sim\cN_{N,k+p'}(0,1)$ by independence, so it follows by properties of stochastic dominance that $\Phi_{\Sigma,k,p}(\theta)\ge\Phi_{\Sigma,k,p'}(\theta)$ for every $0\le\theta\le\pi/2$~\cite{Shaked2007}. The second claim follows by similar logic; if $\theta_1''$ is the largest principal angle between $I_{N,k-j}$ and $\Sigma\Omega$, then $\theta_1\ge\theta_1''$, so stochastic dominance implies $\Phi_{\Sigma,k,p}(\theta)\le\Phi_{\Sigma,k-j,p+j}(\theta)$.

For the third claim, we assume without loss of generality that we delete the last column and row of $\Sigma$ to form $\Sigma^-$. Let $\theta_1,\theta_1^-$ denote the largest principal angles in question corresponding to the singular value matrices $\Sigma,\Sigma^-$ respectively. Form the conformal partitions
\[
\Sigma = \begin{bmatrix} \Sigma^- & 0 \\ 0 & \sigma \end{bmatrix} \qquad
\Omega = \begin{bmatrix} \Omega^- \\ \omega^\top \end{bmatrix}
\]
where $\Sigma^-\in\R^{(N-1)\times(N-1)}$, $\sigma\in\R$, and $\omega\sim\cN_{k+p,1}(0,1)$. Let $R\in\St_{k+p,k+p-1}$ have columns span $\mathrm{null}(\omega^\top)$ so that $\Omega R = \begin{bmatrix}\Omega^-R \\ 0\end{bmatrix}$. Conditioned on $\omega$, we have $\Omega^-R\sim\cN_{N-1,k+p-1}(0,1)$ by~\cite[Thm.~2.3.10]{Gupta2000}. This distribution does not depend on $\omega$ or $R$, hence $\Omega^-R$ has the same distribution without conditioning on $\omega$. Finally observe that $\theta^-$ is the largest principal angle between $I_{N,k}$ (up to permutation of rows) and the column space of $\Sigma\Omega R$, the latter of which lies in the column space of $\Sigma\Omega$. Thus $\theta_1\le\theta_1^-$ and the result follows.
\end{proof}

The next set of results bound $\cJ_{N,k,p}(sI_k)$ from below. Inverting these lower bounds leads to the familiar-looking probability estimates for the approximation error angle $\theta_1$. The bounds for odd $p$ are asymptotically tight near $s=1$, but there should be room for improvement for even $p$.

\begin{lemma}\label{lem:J-scalar}
Let $N,k,p\ge1$ be integers such that $N\ge 2k+p$. Let $\mu=\tfrac{p-1}2$ and $\nu=\tfrac{k(N-k-p)}2$. If $p$ is odd, then
\begin{equation}\label{eq:J-bd-odd}
\cJ_{N,k,p}(sI_k) \ge 1 - \binom{\nu+\mu}{\nu-1}\left(\frac{1-s}s\right)^{\mu+1}, \qquad 0<s\le 1.
\end{equation}
If $p$ is even, then
\begin{equation}\label{eq:J-bd-even}
\cJ_{N,k,p}(sI_k) \ge 1 - \binom{\nu+\mu-\frac12}{\nu-1}\left(\frac{1-s}s\right)^{\mu+1/2}, \qquad 0<s\le 1.
\end{equation}
The binomial coefficient is defined in the standard way if $\nu$ is a half-integer.
\end{lemma}
\begin{proof}
Setting $\Sigma_1=\sqrt{s}\,I_k$, $\Sigma_2=\sqrt{1-s}\,I_{N-k}$, and $\theta=\tfrac\pi4$ yields $\cJ_{N,k,p}(sI_k) = \Phi_{\Sigma,k,p}(\tfrac\pi4)$ by~\cref{thm:cdf}; in this setting $S_\theta = sI_k$. Therefore results related to $\Phi_{\Sigma,k,p}(\theta)$ can also apply to $\cJ_{N,k,p}(sI_k)$. 

We prove~\cref{eq:J-bd-odd} assuming odd $p$. The even case follows from~\Cref{lem:monotone} as $\cJ_{N,k,p}(sI_k)\ge\cJ_{N-1,k,p-1}(sI_k)$, and since $p-1$ is odd we apply the bound for the odd case to the right-hand side with the change $\mu\to\mu-\frac12$.

Let $t=1-s$, and set $\gamma=\tfrac{N-k-p}2$, so $\nu=k\gamma$. Since $p$ is odd,~\Cref{cor:terminate} and homogeneity of the zonal polynomials imply
\begin{equation}\label{eq:sai-1}
\cJ_{N,k,p}(sI_k)
= (1-t)^\nu \sum_{\ell=0}^{k\mu} \sum_{\substack{\kappa\,\vdash\ell\\\kappa_1\le\mu}}
	\frac{(\gamma)_\kappa C_\kappa(I_k)}{\ell!} t^\ell.
\end{equation}
Observe that, by~\cite[Cor.~7.3.5]{Muirhead1982}, the summation on the right-hand side is a partial sum of the series
\[
\sum_{\ell=0}^\infty\sum_{\kappa\,\vdash\ell}\frac{(\gamma)_\kappa C_\kappa(I_k)}{\ell!} t^\ell 
= \oFz(\gamma;t I_k) 
= (1-t)^{-\nu} 
= \sum_{\ell=0}^\infty\frac{(\nu)_\ell}{\ell!}t^\ell.
\]
The condition $\kappa_1\le\mu$ in~\cref{eq:sai-1} is guaranteed when $\ell\le\mu$, hence
\[
	\sum_{\substack{\kappa\,\vdash\ell\\\kappa_1\le\mu}} \frac{(\gamma)_\kappa C_\kappa(I_k)}{\ell!} 
	= \sum_{\kappa\,\vdash\ell} \frac{(\gamma)_\kappa C_\kappa(I_k)}{\ell!} 
	= \frac{(\nu)_\ell}{\ell!},\qquad \forall\,\ell\le\mu.
\]
Because each coefficient is positive, we bound
\begin{equation}\label{eq:sai-2}
	\cJ_{N,k,p}(sI_k) \ge (1-t)^\nu\sum_{\ell=0}^\mu\frac{(\nu)_\ell}{\ell!}t^\ell = s^\nu\tFo(-\mu,\nu;-\mu;1-s),
\end{equation}
where the right-hand side is the univariate Gauss hypergeometric series and is understood in the limiting sense of \cite[Eq.~(15.2.5)]{NIST}. By the univariate reflection formula \cite[Eq.~(15.8.7)]{NIST}, the hypergeometric representation of the beta function \cite[Eq.~(8.17.7)]{NIST}, and elementary relations between the gamma, beta, and rising factorial functions, we rewrite
\begin{multline*}
s^\nu\tFo(-\mu,\nu;-\mu;1-s) = s^\nu\frac{\Gamma(\nu+\mu+1)}{\Gamma(\nu+1)\Gamma(\mu+1)}\tFo(-\mu,\nu;\nu+1;s) \\
= \frac{\nu\Gamma(\nu+\mu+1)}{\Gamma(\nu+1)\Gamma(\mu+1)}\Beta_s(\nu,\mu+1) =  I_s(\nu,\mu+1) = 1-I_{1-s}(\mu+1,\nu),
\end{multline*}
where $I_s(a,b)$ denotes the regularized incomplete beta function \cite[Eq.~(8.17.2)]{NIST}, and we additionally use the transformation \cite[Eq.~(8.17.4)]{NIST}. Since $\nu\ge1$ by assumption, the integral representation \cite[Eq.~(8.17.1)]{NIST} yields the bound
\begin{multline*}
I_{1-s}(\mu+1,\nu) = \frac{1}{\Beta(\mu+1,\nu)}\int_0^{1-s} x^\mu(1-x)^{\nu-1}\,\dd x \\
\le \frac{1}{(\mu+1)\Beta(\mu+1,\nu)}(1-s)^{\mu+1} = \binom{\nu+\mu}{\nu-1}(1-s)^{\mu+1}.
\end{multline*}
Since $1-s\le\frac{1-s}{s}$ for $s\in[0,1]$, inserting this bound back into~\cref{eq:sai-2} yields the desired inequality.

For the exceptional case $k=N-k-p=1$, we have $\nu=\frac12$ and the inequality in the previous display no longer holds. However, the change of variables $y=\frac{x}{1-x}$ in the integral yields
\[
I_{1-s}(\mu+1,\tfrac12) = \frac{1}{\Beta(\mu+1,\frac12)}\int_0^{\frac{1-s}s} y^\mu (1+y)^{-(\mu+3/2)}\,\dd y \le \binom{\mu+\frac12}{\mu+1}\left(\frac{1-s}{s}\right)^{\mu+1},
\]
since $\mu+\frac32>0$ implies $(1+y)^{-(\mu+3/2)}\le1$.
\end{proof}

The other bound we need covers the $p=0$ setting.

\begin{lemma}\label{lem:p0}
Let $N>k\ge1$ be integers. Then
\begin{equation}\label{eq:p0}
\cJ_{N,k,0}(sI_k) \ge 1 - k(N-k)\sqrt{1-s}, \qquad 0\le s\le 1.
\end{equation}
\end{lemma}
\begin{proof}
For $p=0$, we have $\cJ_{N,k,0}(sI_k) = K_{N,k,0}\ s^{k(N-k)/2} \tFo(\tfrac12,\tfrac{N-k}2;\tfrac{N+1}2;sI_k)$.
By definition of the partitional shifted factorial, $(\tfrac12)_\kappa=0$ whenever $\kappa=(\kappa_1,\kappa_2,\dots)$ has $\kappa_2>0$. Using the identity $C_{(\ell)}(I_k) = (\tfrac k2)_\ell / (\tfrac12)_\ell$ \cite[Eq.~(7.2.18)]{Muirhead1982}, we have
\[
\tFo(\tfrac12,\tfrac{N-k}2;\tfrac{N+1}2;sI_k) = \sum_{\ell=0}^\infty \frac{(\tfrac{N-k}2)_\ell(\tfrac k2)_\ell}{\ell!(\tfrac{N+1}2)_\ell} s^\ell = \tFo(\tfrac k2,\tfrac{N-k}2;\tfrac{N+1}2;s),
\]
where the right-hand side is the univariate hypergeometric function. As $\cJ_{N,k,0}(I_k)=1$, the Gauss formula \cite[Eq.~(15.4.20)]{NIST} implies
\[
K_{N,k,0} = \frac{1}{\tFo(\tfrac k2,\tfrac{N-k}2;\tfrac{N+1}2;1)} = \frac{\Gamma(\frac{k+1}2)\Gamma(\frac{N-k+1}2)}{\Gamma(\frac{N+1}2)\Gamma(\frac12)}.
\]

Let $a=\tfrac k2$, $b=\tfrac{N-k}2$, $c=a+b+\tfrac12=\tfrac{N+1}2$, and $F(s):=K_{N,k,0}\,\tFo(a,b;c;s)$. We have $1-F(s) = \int_s^1 F'(t)\,\dd t$ since $F(1)=1$. We aim to show that $0\le F'(t)\le ab(1-t)^{-1/2}$ for this implies $1-F(s)\le \frac{k(N-k)}2\sqrt{1-s}$. By the differentiation formula \cite[Eq.~(15.5.1)]{NIST}, we have
\[
F'(s) = K_{N,k,0}\frac{ab}{c}\tFo(a+1,b+1;c+1;s) \ge 0.
\]
Recall that $(1-t)^{-1/2} = \sum_{\ell=0}^\infty \frac1{\ell!}(\tfrac12)_\ell\,t^\ell$, so it suffices for us to show that
\[
K_{N,k,0}\frac{(a+1)_\ell(b+1)_\ell}{c(c+1)_\ell} \le (\tfrac12)_\ell
\]
for all integers $\ell\ge0$, since all coefficients in the power series are non-negative. This is equivalent to showing that
\begin{equation}\label{eq:p0-0}
\frac{(a+1)_\ell(b+1)_\ell}{(\tfrac12)_\ell(c+1)_\ell} \le \frac{\Gamma(c+1)\Gamma(\tfrac12)}{\Gamma(a+\tfrac12)\Gamma(b+\tfrac12)},
\end{equation}
using the identity $c\Gamma(c)=\Gamma(c+1)$. Let $M_\ell$ denote the expression on the left-hand side; observe the right-hand side does not depend on $\ell$. We have
\[
\frac{M_{\ell+1}}{M_\ell} = \frac{(a+\ell+1)(b+\ell+1)}{(\ell+\tfrac12)(c+\ell+1)},
\]
and basic algebraic manipulation shows the numerator dominates the denominator for $a,b\ge0$.  Thus $M_\ell$ is increasing in $\ell$, so it suffices to show~\cref{eq:p0-0} in the limit as $\ell\to\infty$. Using the identity $(z)_\ell = \frac{\Gamma(z+\ell)}{\Gamma(z)}$ and the gamma function asymptotics \cite[Eq.~(5.11.12)]{NIST}, we find
\[
\lim_{\ell\to\infty} M_\ell = \frac{\Gamma(c+1)\Gamma(\tfrac12)}{\Gamma(a+1)\Gamma(b+1)}
\]
since $a+b-c+\tfrac12=0$. 

To prove~\cref{eq:p0-0} it remains to show that $\Gamma(a+1)\Gamma(b+1)\ge\Gamma(a+\tfrac12)\Gamma(b+\tfrac12)$. Let $\psi(x):=\frac{\dd}{\dd x}(\log\Gamma(x))$ denote the digamma function. For real $x>0$, $\psi(x)$ is strictly increasing \cite[Eq.~(5.15.1)]{NIST}, hence
\[
\psi(x+\tfrac12)-\psi(x) = \frac{\dd}{\dd x}\left(\log\frac{\Gamma(x+\tfrac12)}{\Gamma(x)}\right) > 0
\]
which implies $x\mapsto\Gamma(x+\tfrac12)/\Gamma(x)$ is increasing for $x>0$. Under the assumption $k(N-k)\ge2$ in addition to $N>k\ge1$, we have $a,b\ge\tfrac12$, and at least one of $a$ or $b$ is $\ge1$. Therefore,
\[
\frac{\Gamma(a+1)\Gamma(b+1)}{\Gamma(a+\tfrac12)\Gamma(b+\tfrac12)} \ge \frac{\Gamma(\tfrac32)\Gamma(2)}{\Gamma(1)\Gamma(\tfrac32)} = 1.
\]
This completes the proof that $1-F(s)\le \frac{k(N-k)}2\sqrt{1-s}$ when $k(N-k)\ge2$.

Next, observe that $s^{k(N-k)/2} \ge 1-\frac{k(N-k)}2\sqrt{1-s}$ for $k(N-k)\ge2$ and $s\in[0,1]$. This follows from the fact that $s^{k(N-k)/2}\ge1-\frac{k(N-k)}2(1-s)$ by convexity and taking the tangent line at $x=1$, then using $1-s\le\sqrt{1-s}$ for $s\in[0,1]$. We conclude
\[
\cJ_{N,k,0}(sI_k) \ge \max\left(0,1-\frac{k(N-k)}2\sqrt{1-s}\right)^2 \ge 1 - k(N-k)\sqrt{1-s},
\]
as desired.

The remaining exceptional case where $k(N-k)<2$ is $k=1$, $N=2$. The identity \cite[Eq.~(15.4.4)]{NIST} yields $\cJ_{2,1,0}(sI_k) = \frac2\pi\arcsin\sqrt{s}$, which dominates $1-\sqrt{1-s}$ as is easily verified by transforming $s\mapsto\sin^2\theta$ for some $\theta\in[0,\tfrac\pi2]$.
\end{proof}

\Cref{lem:J-scalar,lem:p0} allow us to derive probabilistic estimates for the largest principal angle $\theta_1$ when $\Sigma_1,\Sigma_2$ are scalar matrices. By inverting the lower bounds on the CDF, we obtain upper bounds on the angular error in terms of the desired failure tolerence.

\begin{proposition}\label{prop:scalar-bd}
Assume the hypotheses of~\cref{thm:cdf} with $\rank(A)>k+p$ and $N\ge2k+p$. Let $\mu=\frac{p-1}2$, $\nu=\frac{k(N-k-p)}{2}$, and $0<\delta\le1$, and define
\begin{equation}\label{eq:constants}
\begin{split}
C_1 = \left[\delta^{-1}\binom{\nu+\mu}{\nu-1}\right]^{\frac1{p+1}}, &\qquad
C_2 = \delta^{-\frac1{p+1}}\sqrt{e\cdot\frac{k(N-k)-p(k-1)-1}{p+1}}, \\
C_3 = \left[\delta^{-1}\binom{\nu+\mu-\frac12}{\nu-1}\right]^{\frac1p}, &\qquad
C_4 = \delta^{-\frac1p}\sqrt{e\cdot\frac{k(N-k)-p(k-1)-2}{p}}, \\
C_5 = \delta^{-1}k(N-k). \qquad\quad\ \,&
\end{split}
\end{equation}
If $\Sigma_1=aI_k$ and $\Sigma_2=bI_{N-k}$ are scalar matrices with $a,b>0$, then each of the following hold with probability $\ge1-\delta$:
\begin{enumerate}[leftmargin=8mm]
\item If $p\ge1$ is odd, then
\begin{textwide}
\begin{equation}\label{eq:scalar-p1-odd}
\sin\theta_1 \le \frac{(b/a)C_1}{\sqrt{1+(b/a)^2C_1^2}} 
	\le \frac{(b/a)C_2}{\sqrt{1+(b/a)^2C_2^2}}.
\end{equation}
\end{textwide}

\item If $p\ge2$ is even, then
\begin{textwide}
\begin{equation}\label{eq:scalar-p1-even}
\sin\theta_1 \le \frac{(b/a)C_3}{\sqrt{1+(b/a)^2C_3^2}} 
	\le \frac{(b/a)C_4}{\sqrt{1+(b/a)^2C_4^2}}.
\end{equation}
\end{textwide}

\item If $p=0$, then
\begin{textwide}
\begin{equation}\label{eq:scalar-p0}
\sin\theta_1 \le \frac{(b/a)C_5}{\sqrt{1+(b/a)^2C_5^2}}.
\end{equation}
\end{textwide}
\end{enumerate}
\end{proposition}
\begin{proof}
When $\Sigma_1,\Sigma_2$ are scalar, the integrand of~\cref{eq:cdf} no longer depends on $Q_1,H_1$, so $\Phi_{\Sigma,k,p}(\theta) = \cJ_{N,k,p}(s(\theta)I_k)$, where $s(\theta)=(1+\cot^2(\theta)(b/a)^2)^{-1}$. By selecting $\theta_\delta$ such that $\Phi_{\Sigma,k,p}(\theta_\delta)\ge1-\delta$, we guarantee that $\theta_1\le\theta_\delta$ with probability $\ge1-\delta$. If $f:(0,1)\to\R$ is an invertible function with $\cJ_{N,k,p}(sI_k)\ge 1-f(s)$, then it suffices to choose $\theta_\delta$ such that $s(\theta_\delta) = f^{-1}(\delta)$. 

Applying the inequality $1-s\le\frac{1-s}{s}$ on $s\in[0,1]$ to~\cref{eq:p0}, we find on the basis of~\Cref{lem:J-scalar,lem:p0} that for each of the three parameter regimes for $p$ considered in this result, one has the lower bound $\cJ_{N,k,p}(sI_k)\ge 1-f(s)$, where $f$ is of the form $f(s) = C(\frac{1-s}{s})^\alpha$ for some constant $C$ and $\alpha>0$. Therefore we choose $\theta_\delta$ such that
\[
s(\theta_\delta) = f^{-1}(\delta) = \frac{1}{1+\left(\frac\delta C\right)^{1/\alpha}}.
\]
Note that $s(\theta)$ is also invertible, leading to the choice $\theta_\delta = \arctan(\frac ba(\frac C\delta)^{1/2\alpha})$. For $p\ge1$ odd, $C=\binom{\nu+\mu}{\nu-1}$ and $\alpha=\mu+1$. For $p\ge2$ even, $C=\binom{\nu+\mu-1/2}{\nu-1}$ and $\alpha=\mu+\frac12$. For $p=0$, $C=k(N-k)$ and $\alpha=\frac12$. Finally, we apply the identity $\sin\theta=\frac{\tan\theta}{\sqrt{1+\tan^2\theta}}$ to obtain the primary inequalities in~\cref{eq:scalar-p1-odd,eq:scalar-p1-even}, as well as~\cref{eq:scalar-p0}.

The secondary inequalities in~\cref{eq:scalar-p1-odd,eq:scalar-p1-even} follow from the binomial coefficient identity $\binom nm = \binom{n}{n-m}$ and the estimate $\binom nm<(\frac{en}{m})^m$, which holds if $n>m-1\ge0$ and $m$ is an integer. It is easy to verify these conditions are satisfied in the relevant parameter regimes.
\end{proof}

\subsection{Probability estimates in terms of the Frobenius singular value ratio}
Using the tools from~\cref{ss:auxiliary}, we now set out to prove our main bound,~\cref{thm:Fgap-bound} below, on the RRF's subspace approximation error in terms of the Frobenius singular value ratio $\xi_k$. To outline the proof, we begin by representing $\Phi_{\Sigma,k,p}(\theta)$ of~\cref{eq:cdf} as a mixture of CDF's of certain beta random matrices. We exploit a convexity property of the density function of these beta random matrices in certain parameter regions, in conjunction with a generalized version of the semi-classical Anderson's theorem on monotonicity properties of symmetric unimodal distributions, to bound $\Phi_{\Sigma,k,p}(\theta)$ from below by a quantity that is easy to work with; from there elementary inequalities, in addition to~\cref{lem:Etr-eq} below, yield the bound~\cref{eq:Fgap-bound}.

As far as we are aware, estimates of the form~\cref{eq:Fgap-bound}, with the Frobenius singular value ratio, cannot be obtained by simple modifications of proof techniques for similar estimates. The usual argument involves bounding the spectral norm of the matrix $\Sigma_2V_2^\top\Omega(\Sigma_1V_1^\top\Omega)^\dagger$, in the notation of~\cref{eq:A-partition}, by considering the matrices $\widetilde\Omega_1:=\Sigma_1V_1^\top\Omega$ and $\widetilde\Omega_2:=\Sigma_2V_2^\top\Omega$ separately and bounding their smallest and largest singular values respectively by means of Gaussian concentration inequalities. While our approach, by using the exact CDF of the angular error, seems fundamentally different, it should really be viewed as a sharpening of the probabilistic analysis for the matrix $\widetilde\Omega_2\widetilde\Omega_1^\dagger$. In particular, given probabilistic bounds on $\|\widetilde\Omega_2\|$ and $\|\widetilde\Omega_1^\dagger\|$, the resulting union bound on $\|\widetilde\Omega_2\widetilde\Omega_1^\dagger\|$ is tight only if the largest right singular vector of $\widetilde\Omega_2$ aligns well with the smallest right singular vector of $\widetilde\Omega_1$; this is rarely the case due to the isotropicity of the Gaussian matrix. Our CDF considers the spectral norm of $\widetilde\Omega_2\widetilde\Omega_1^\dagger$ directly by viewing it essentially as (half of) a generalized beta random matrix, deriving its distribution exactly, and saving the inequalities leading to the estimate for the very end. Maintaining exactness in our computations for as long as possible seems to be critical for proving a strong estimate that holds in all parameter regimes.

\begin{lemma}\label{lem:Etr-eq}
Set $k\le n$. For fixed $X\in\cS_k^+$, $Y\in\cS_n^+$ and $H_1\sim\Unif(\St_{n,k})$, we have $\E[\tr(XH_1^\top YH_1)] = \tfrac1n\tr(X)\tr(Y)$, where $\tr(\cdot)$ denotes the trace.
\end{lemma}
\begin{proof}
Since $C_{(1)}(X)=\tr(X)$, the equality immediately follows from the mean-value property of zonal polynomials; see~\cite[Thm.~7.2.5]{Muirhead1982} and~\cite[Thm.~1]{Shimizu2021}. 
\end{proof}

We are now ready to present our main result in this section: a probabilistic error estimate for the RRF in terms of the Frobenius singular value ratio~\cref{eq:Fgap}.

\begin{theorem}\label{thm:Fgap-bound}
Assume the hypotheses of~\cref{thm:cdf} with $\rank(A)>k+p$. If $p\ge1$ and $N\ge 2k+p$, then with probability $\ge1-\delta$,
\begin{equation}\label{eq:Fgap-bound}
\sin\theta_1 \le \frac{\xi_k C_0}{\sqrt{1+\xi_k^2C_0^2}},
\end{equation}
where $C_0=\sqrt{\frac{k(N-k-1)}{2\delta}}$ and $\xi_k$ is the Frobenius singular value ratio from~\cref{eq:Fgap}. 

If $p=0$ and $N\ge 2k+p$, the same bound holds replacing $C_0$ with $C_5$ defined in~\cref{eq:constants}. If $N<2k+p$, then the same bounds for $p=0$ and $p\ge1$ hold  replacing $N$ everywhere with $2k+p$ and padding $\Sigma_2$ with zeros as in~\Cref{thm:cdf}, part (3).
\end{theorem}

\begin{proof}
Without loss of generality, we assume $N\ge 2k+p$ and that $A$ has full rank so that $S_\theta$ of~\cref{eq:S} takes the form~\cref{eq:tilde-S}; the rank-deficient case and the $N<2k+p$ case follows from the continuity argument and the padding argument as in the proof of~\cref{thm:cdf}, respectively. We begin with the observation that $\Phi_{\Sigma,k,p}(\theta)$ can be expressed as the CDF\ of a certain mixture distribution. Namely,
\begin{equation}\label{eq:beta-rep}
	\Phi_{\Sigma,k,p}(\theta) = \E_{Q_1,H_1}[\P_U(S_\theta \succ  U)],
\end{equation}
where $U\sim\beta_k^\mathrm{I}(\tfrac{N-k-p}2,\tfrac{k+p}2)$ has the matrix-variate beta type I distribution of~\cite[Def.~5.2.1]{Gupta2000}, and is independent of $Q_1,H_1$. The representation~\cref{eq:beta-rep} can be readily verified using the formula for the CDF\ of $U$~\cite[p.~166]{Gupta2000}. Now the map $x\mapsto \tfrac1x-1$ is operator decreasing on $[0,1]$, so by~\cite[Thms.~5.3.6--7]{Gupta2000}, we have
\[
	\Phi_{\Sigma,k,p}(\theta) = \E_{Q_1,H_1}[\P_V(X_\theta\prec V)] = \E_{Q_1,H_1}\left[\int_{V\succ 0}\phi(V)\,\1\{X_\theta \prec  V\}\,(\dd V)\right],
\]
where $X_\theta := \cot^2(\theta)\Sigma_1^{-1}(Q_1^\top(H_1^\top\Sigma_2^2H_1)^{-1}Q_1)^{-1}\Sigma_1^{-1}$ and $V\sim\beta_k^{\mathrm{II}}(\tfrac{k+p}2,\tfrac{N-k-p}2)$ has the matrix-variate beta type II distribution with density function~\cite[Def.~5.2.2]{Gupta2000}\footnote{\,This distribution is, confusingly, also sometimes referred to as the MANOVA ensemble. Another name for it is the matrix $F$-distribution~\cite{Mulder2018}.}
\begin{equation}\label{eq:beta-ii}
	\phi(V) = \frac{1}{\Beta_k(\tfrac{k+p}2,\tfrac{N-k-p}2)}|V|^{\tfrac{p-1}2}|I+V|^{-\tfrac N2}.
\end{equation}

Assume for now that $p>1$. Fix small $\epsilon>0$ and define auxiliary matrices $\widehat X_\theta\in\cS_{k+p-1}^{++}$ and $\widehat V\sim\beta_{k+p-1}^{\mathrm{II}}(\tfrac{k+p}2,\tfrac{N-k-1}2)$, independent of $V$, which we conformally partition into $k\times k$ and $(p-1)\times (p-1)$ diagonal blocks:
\[
	\widehat V = \begin{bmatrix}\widehat V_{11} & \widehat V_{12} \\ \widehat V_{21} & \widehat V_{22}\end{bmatrix}, \qquad \widehat X_\theta = \begin{bmatrix} X_\theta & 0 \\ 0 & \frac{\epsilon}{p-1} I_{p-1} \end{bmatrix}.
\]
First, $\widehat X_\theta\prec \widehat V$ implies $X_\theta\prec \widehat V_{11}$, since all diagonal blocks of SPD matrices must also be SPD. Secondly, by~\cite[Thm.~5.3.11]{Gupta2000}, $\widehat V_{11}$ and $V$ are identically distributed. It follows that for fixed $X_\theta$, we have $\P_V(X_\theta\prec V)\ge\P_{\widehat V}(\widehat X_\theta\prec \widehat V)$.

Next, we would like to apply~\cite[Thm.~1]{DasGupta1976}, a generalized version of Anderson's theorem, to obtain a monotonicity property for
\[
	\P_{\widehat V}(\widehat X_\theta\prec \widehat V) = \int_{\widehat V\succ 0}\phi(\widehat V)\1\{\widehat X_\theta\prec \widehat V\}(\dd \widehat V)
\]
with respect to $\widehat X_\theta$. This requires us to show that $f_1(\widehat X_\theta,\widehat V)=\phi(\widehat V)$ and $f_2(\widehat X_\theta,\widehat V)=\1\{\widehat X_\theta\prec \widehat V\}$ are invariant with respect to a group action on $(\widehat X_\theta,\widehat V)\in\cS_{k+p-1}^{++}\times\cS_{k+p-1}^{++}$ satisfying certain measure-preserving properties, and that $f_1,f_2$ are quasiconvex and quasiconcave functions respectively, where quasiconvex (quasiconcave) means that every lower (upper) level set is convex. For quasiconvexity, recall that by~\cref{eq:beta-ii}, the density function of $\widehat V$ is
\[
	\phi(\widehat V) = \frac{1}{\Beta_{k+p-1}(\tfrac{k+p}2,\tfrac{N-k-1}2)}|I_{k+p-1}+\widehat V|^{-\tfrac{N+p-1}2}, \qquad \widehat V\succ 0.
\]
The determinant is log-concave for SPD matrices~\cite[p.~73]{Boyd2004}, so $\phi(\widehat V)$ is log-convex and thus also quasiconvex~\cite[p.~104]{Boyd2004}. On the other hand, the set $\{\widehat X_\theta\prec \widehat V\}$ is jointly convex in $(\widehat X_\theta,\widehat V)\in\cS_{k+p-1}^{++}\times\cS_{k+p-1}^{++}$, so the indicator function $\1\{\widehat X_\theta\prec \widehat V\}$ is trivially quasiconcave (its upper level sets are either the empty set, the entire space $\cS_{k+p-1}^{++}\times\cS_{k+p-1}^{++}$, or $\{\widehat X_\theta\prec \widehat V\}$ itself).

For the group action, let $\Or_{k+p-1}\times\Or_{k+p-1}$ act on $(\widehat X_\theta,\widehat V)\in\cS_{k+p-1}^{++}\times\cS_{k+p-1}^{++}$ by conjugation component-wise; this action is Lebesgue measure-preserving since each $\Or_{k+p-1}$ acts independently on one copy of $\cS_{k+p-1}^{++}$ by invertible linear transformations with determinant 1. We consider $\Or_{k+p-1}$ as the diagonal subgroup which acts by conjugating both $\widehat X_\theta$ and $\widehat V$ by the same orthogonal matrix. Both $f_1$ and $f_2$ are invariant under this subgroup action since $\phi(R^\top\widehat VR) = \phi(\widehat V)$ and $\1\{R^\top \widehat X_\theta R\prec R^\top\widehat VR\} = \1\{\widehat X_\theta\prec\widehat V\}$ for any $R\in\Or_{k+p-1}$. As a result, the hypotheses of~\cite[Thm.~1]{DasGupta1976} are satisfied, and we conclude that $\P(\widehat X_\theta\prec \widehat V)$ decreases whenever $\widehat X_\theta$ is replaced with any SPD matrix that lies within the convex hull of its orbit with respect to special orthogonal conjugation. The convex hull contains all diagonal matrices of permutations of the eigenvalues of $\widehat X_\theta$, so in particular it contains the scalar matrix where all diagonal elements are $\frac1{k+p-1}\tr(\widehat X_\theta)$, the arithmetic mean of its eigenvalues. Moreover, we have by construction $\tr(\widehat X_\theta)=\tr(X_\theta)+\epsilon$. In conjunction with~\cref{eq:beta-rep}, we have therefore shown that 
\[
	\Phi_{\Sigma,k,p}(\theta) \ge \E_{Q_1,H_1}\left[\cJ_{N+p-1,k+p-1,1}\left(\frac{1}{1+\tfrac1{k+p-1}(\tr(X_\theta)+\epsilon)}I_{k+p-1}\right)\right].
\]
Note by definition that $\cJ_{N,k,1}(S) = |S|^{(N-k-1)/2}$, hence
\begin{equation}\label{eq:Fgap-0}
	\Phi_{\Sigma,k,p}(\theta) \ge \E_{Q_1,H_1}\left[\left(\frac{1}{1+\tfrac1{k+p-1}(\tr(X_\theta)+\epsilon)}\right)^{\tfrac{(k+p-1)(N-k-1)}2}\right].
\end{equation}

We make two remarks. First, the integrand in~\cref{eq:Fgap-0} is convex as a function of $\tr(X_\theta)$, so by Jensen's inequality we can move the expectation operator to the denominator of the fraction. Second, the operator convexity of the map $x\mapsto\frac1x$ on $(0,\infty)$ implies $\lambda^\downarrow((Q_1^\top(H_1^\top\Sigma_2^2H_1)^{-1}Q_1)^{-1}) \le \lambda^\downarrow(Q_1^\top H_1^\top\Sigma_2^2H_1Q_1),$ for each $Q_1\in\St_{k+p,k},H_1\in\St_{N-k,k+p}$, by~\cite[Thm.~V.2.1]{Bhatia1997}. Two applications of~\cref{lem:Etr-eq} yield $\E[\tr(X_\theta)] \le \tfrac1{N-k}\cot^2(\theta)\tr(\Sigma_1^{-2})\tr(\Sigma_2^2)$.
Altogether, we obtain
\begin{multline*}
	\Phi_{\Sigma,k,p}(\theta) \ge \left(1+C\cot^2(\theta)+\tfrac{\epsilon}{k+p-1}\right)^{-(k+p-1)(N-k-1)/2} \\
	= \cJ_{N+p-1,k+p-1,1}\left(\frac{1}{1+C\cot^2(\theta)+\tfrac{\epsilon}{k+p-1}}I_{k+p-1}\right),
\end{multline*}
where $C=\frac{\tr(\Sigma_1^{-2})\tr(\Sigma_2^2)}{(k+p-1)(N-k)}$. This holds for all sufficiently small $\epsilon>0$, so the same bound holds taking $\epsilon\to0$ by continuity of $\cJ_{N+p-1,k+p-1,1}$. To conclude, we apply the first inequality of~\cref{eq:scalar-p1-odd}, substituting $N+p-1$, $k+p-1$, $1$, $\sqrt{(k+p-1)/\tr(\Sigma_1^{-2})}$, and $\sqrt{\tr(\Sigma_2^2)/(N-k)}$ for $N$, $k$, $p$, $a$, and $b$ in the statement of~\cref{prop:scalar-bd}, respectively, noting that $b/a = \xi_k\sqrt{k/(k+p-1)}$. With these parameters in the notation of~\cref{eq:scalar-p1-odd}, $(b/a)C_1 = \xi_k\sqrt{k(N-k-1)/(2\delta)}$. This completes the proof for $p>1$.

If $p=0$ or 1, the argument is more direct. In this setting, without defining auxiliary matrices $\widehat X_\theta,\widehat V$, the density $\phi(V)$ of~\cref{eq:beta-ii} is already log-convex, so we may proceed immediately to~\cref{eq:Fgap-0} if $p=1$, or if $p=0$, obtain
\[
	\Phi_{\Sigma,k,0}(\theta) \ge \E_{Q_1,H_1}\left[\cJ_{N,k,0}\left(\frac{1}{1+\tfrac1k\tr(X_\theta)}I_k\right)\right]. 
\]
Recall from~\Cref{lem:p0} that $\cJ_{N,k,0}(sI_k) = K_{N,k,0}\,s^{k(N-k)/2}\tFo(\frac k2,\frac{N-k}2;\frac{N+1}2;s)$. If $k(N-k)\ge2$, then $s^{k(N-k)/2}$ is convex, positive, and increasing. The same holds for the hypergeometric function since its power series expansion around $s=0$ has non-negative coefficients and radius of convergence 1. Thus $\cJ_{N,k,0}(sI_k)$ is convex and increasing. Since $x\mapsto\frac1{1+x}$ is convex, it follows that $\cJ_{N,k,0}(\frac{1}{1+\frac1k\tr(X_\theta)}I_k)$ is convex in $\tr(X_\theta)$. On the other hand, it is easily verified for $k=N-k=1$ that $\cJ_{2,1,0}(\frac{1}{1+\frac1k\tr(X_\theta)}I_k)=\frac2\pi\arcsin(\frac{1}{\sqrt{1+\frac1k\tr(X_\theta)}})$ is also convex in $\tr(X_\theta)$. We may then, as in the $p>1$ case, apply Jensen's inequality,~\Cref{lem:Etr-eq}, and~\Cref{prop:scalar-bd} to complete the proof.
\end{proof}

\subsection{Worst-case gap analysis}
A second interesting consequence of~\cref{thm:cdf} is that the RRF is more accurate, in the sense of stochastic dominance, for matrices whose leading singular values are larger and whose tail singular values are smaller (see~\cref{thm:monotone} below). While this seems intuitive, we emphasize that the stochastic dominance aspect of our claim is very strong. It implies that out of the entire family of matrices $\cA_{k,\rho}=\{A\in\R^{m\times n}:\frac{\sigma_{k+1}(A)}{\sigma_k(A)}=\rho\}$ whose singular value gap at index $k$ is prescribed to be $\rho\in[0,1]$, the RRF performs the worst when the leading and tail singular values have a flat profile, namely, when $\sigma_1=\cdots=\sigma_k$ and $\sigma_{k+1}=\cdots=\sigma_n$. In other words, existing estimates that depend only on the singular value gap $\rho_k$ are worst-case bounds.

\begin{theorem}\label{thm:monotone}
Assume the hypotheses of~\cref{thm:cdf} with $\rank(A)>k+p$. For some $\widehat\sigma_1\ge\dots\ge\widehat\sigma_N\ge0$ and $\sigma_1\ge\dots\ge\sigma_N\ge0$, set  $\Sigma=\diag(\sigma_1,\dots,\sigma_N)$ and $\widehat\Sigma = \diag(\widehat\sigma_1,\dots,\widehat\sigma_N)$. If $\widehat\sigma_j\ge\sigma_j$ for $1\le j\le k$, and $\widehat\sigma_j\le\sigma_j$ for $j>k$, then $\Phi_{\Sigma,k,p}(\theta)\le\Phi_{\widehat\Sigma,k,p}(\theta)$ for all $0\le\theta\le\tfrac\pi2$.
\end{theorem}
\begin{proof}
Set $\widehat\Sigma_1=\diag(\widehat\sigma_1,\dots,\widehat\sigma_k)$ and $\widehat\Sigma_2=\diag(\widehat\sigma_{k+1},\dots,\widehat\sigma_N)$. It is easy to see that for fixed $Q,Q_1,H_1$, the eigenvalues of $S_\theta$ of~\cref{eq:S} increase when $\Sigma_1$ is replaced with the larger $\widehat\Sigma_1$. The eigenvalues of $S_\theta$ similarly increase when $\Sigma_2$ is replaced with the smaller $\widehat\Sigma_2$. By~\cref{lem:J-monotone}, $\Phi_{\Sigma,k,p}(\theta)$ is pointwise increasing when the eigenvalues of $S_\theta$ increase.
\end{proof}

In such a worst-case regime, we can derive probabilistic bounds analogous to those of Saibaba~\cite[Thm.~6]{Saibaba2019}. We remark that compared to~\cref{thm:Fgap-bound}, the constant is much smaller as a function of $\delta$, since here $\delta$ is raised to the power of $-1/(p+1)$ instead of $-1/2$.

\begin{theorem}\label{thm:gap-bound}
Assume the hypotheses of~\cref{thm:cdf} with $\rank(A)>k+p$. Then $\sin\theta_1$ satisfies the bounds~\cref{eq:scalar-p1-odd,eq:scalar-p1-even,eq:scalar-p0}, replacing $b/a=\rho_k$, with probability $\ge1-\delta$.
\end{theorem}
\begin{proof}
Let $\rho\equiv\rho_k(A)$ be the singular value ratio of the matrix under consideration. It suffices to prove the statement instead for the matrix $A_\text{flat}$ with singular values $\Sigma_1=I_k$ and $\Sigma_2=\rho I_{N-k}$ since $A_\text{flat}$ is the worst-case matrix for singular subspace approximation given a singular value ratio of size $\rho$, by~\cref{thm:monotone}. But $A_\text{flat}$ immediately satisfies the conditions of~\cref{prop:scalar-bd}, and the result follows.
\end{proof}

More generally, even if the exact singular values of the matrix of interest are unknown, one can still bound the RRF approximation error given partial information on the singular values. Specifically, if one knows that $\Sigma_1\succeq\widetilde\Sigma_1$ and $\Sigma_2\preceq\widetilde\Sigma_2$ for some known $\widetilde\Sigma_1,\widetilde\Sigma_2$, then~\cref{thm:Fgap-bound,thm:monotone} together yield a bound on $\sin\theta_1$ in terms of the Frobenius singular value ratio with respect to $\widetilde\Sigma_1,\widetilde\Sigma_2$. This idea applies, for instance, in the setting where one knows a priori that there is a large gap between $\sigma_k$ and a far-off singular value $\sigma_\ell$, with $\ell>k$, but nothing else about the singular values.

\begin{figure}
\centering
\begin{minipage}{0.47\textwidth}
\begin{overpic}[width=\textwidth]{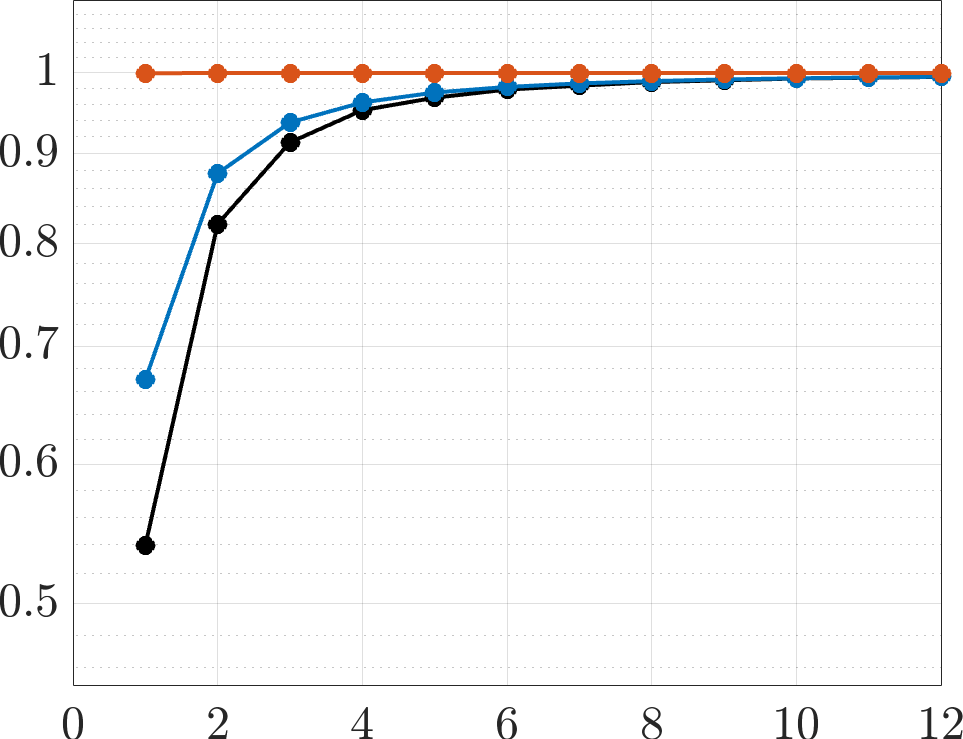}
\put(50,-5){\footnotesize $k$}
\put(20,71){\scriptsize\color{MLorange}Saibaba Thm 6}
\put(12.5,42){\rotatebox{72}{\scriptsize\color{MLblue}Thm 4.7}}
\put(17.5,25){\rotatebox{77}{\scriptsize empirical}}
\put(40,30){\parbox[l]{3cm}{\footnotesize$\sin\theta_1$ falls below black line 95\% of the time}}
\end{overpic}
\vspace*{0mm}
\end{minipage}
\hfill
\begin{minipage}{0.47\textwidth}
\begin{overpic}[width=\textwidth]{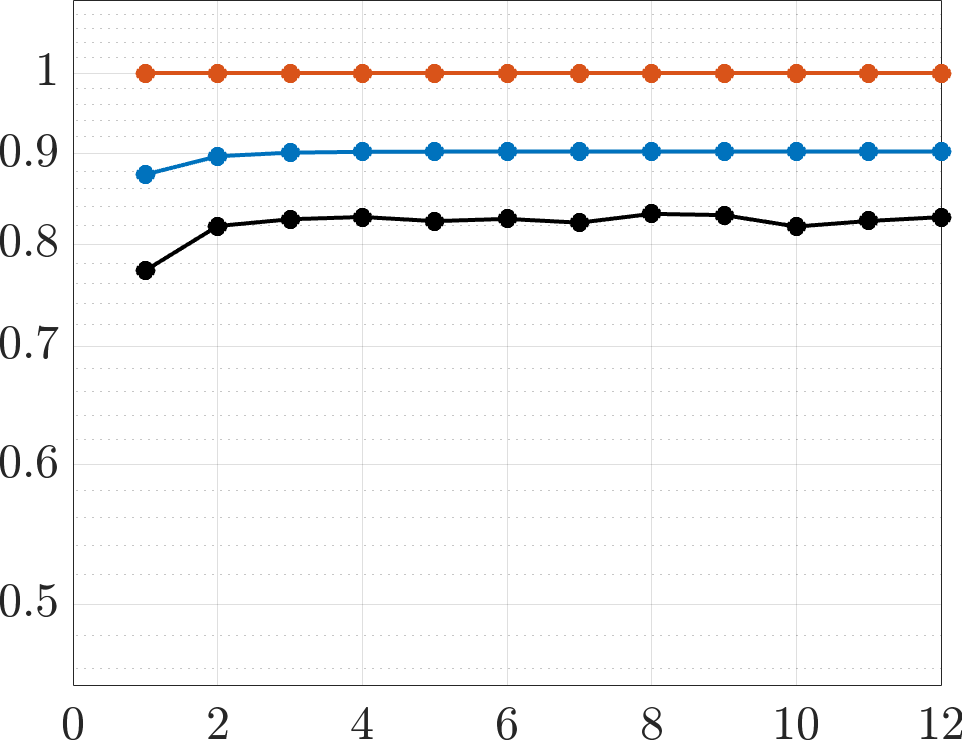}
\put(50,-5){\footnotesize $k$}
\put(20,71){\scriptsize\color{MLorange}Saibaba Thm 6}
\put(20,63){\scriptsize\color{MLblue}Thm 4.7}
\put(24,49.5){\scriptsize empirical}
\put(40,30){\parbox[l]{3cm}{\footnotesize$\sin\theta_1$ falls below black line 95\% of the time}}
\end{overpic}
\vspace*{0mm}
\end{minipage}
\\ \vspace{1mm}
\begin{minipage}{0.47\textwidth}
\begin{overpic}[width=\textwidth]{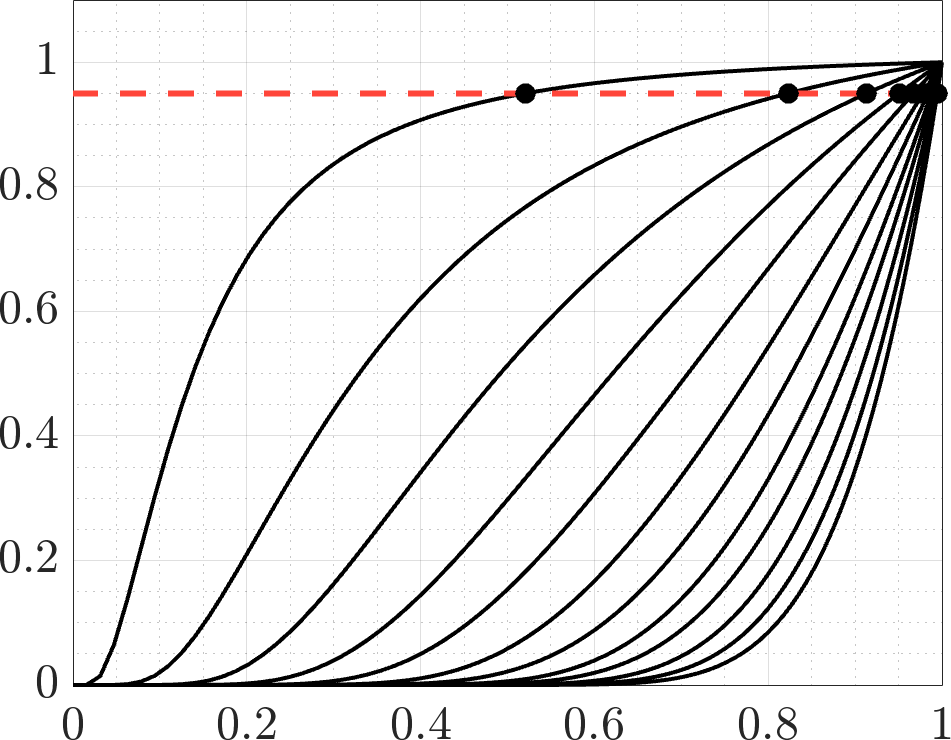}
\put(50,-5){\footnotesize $\sin\theta$}
\put(64.5,71.2){\rotatebox{3}{\scriptsize exact CDFs~\cref{eq:cdf}}}
\put(20,70){\scriptsize\color{myred}$95\%$}
\put(13,30){\rotatebox{73}{\scriptsize $k=1$}}
\put(27.5,30){\rotatebox{57}{\scriptsize $k=2$}}
\put(41.5,30){\rotatebox{50}{\scriptsize $k=3$}}
\put(53.5,30){\rotatebox{50}{\scriptsize $k=4$}}
\put(62,30){\rotatebox{52}{\scriptsize $k=5$}}
\put(89.5,22){\rotatebox{73}{\scriptsize $k=12$}}
\end{overpic}
\vspace*{0mm}
\subcaption{Slow-decaying singular values, $p=1$.}
\end{minipage}
\hfill
\begin{minipage}{0.47\textwidth}
\begin{overpic}[width=\textwidth]{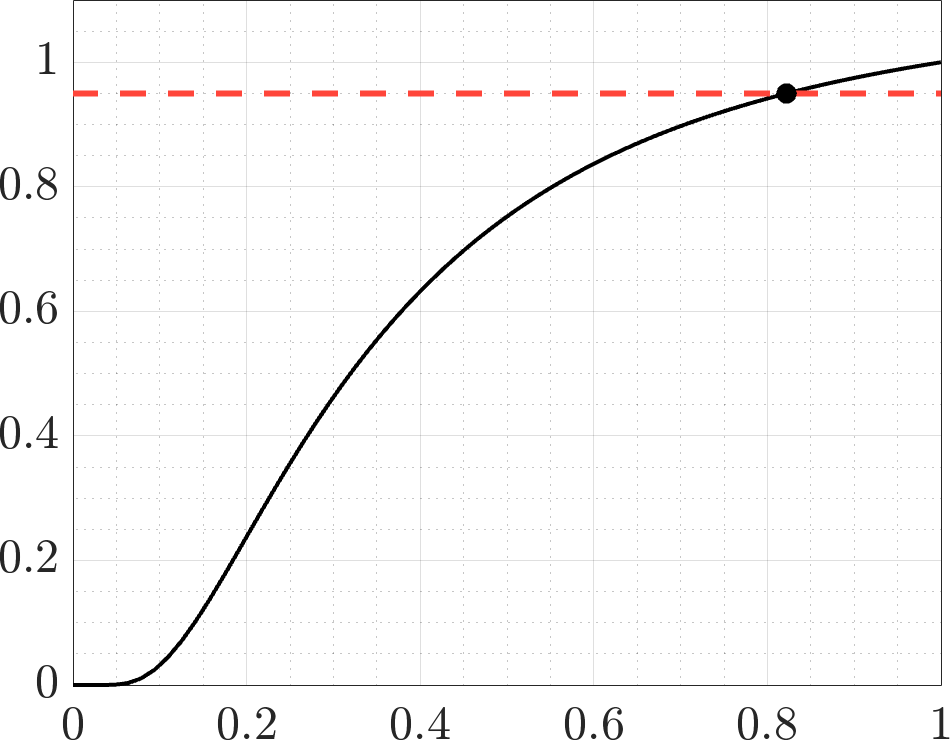}
\put(50,-5){\footnotesize $\sin\theta$}
\put(20,70){\scriptsize\color{myred}$95\%$}
\put(34,30){\rotatebox{55}{\scriptsize $k=8$}}
\put(18,16){\rotatebox{58}{\scriptsize exact CDF~\cref{eq:cdf}}}
\end{overpic}
\vspace*{0mm}
\subcaption{Fast-decaying singular values, $p=1$.}
\end{minipage}
\caption{{\em Top row}: Error estimates for the RRF approximation in terms of $\sin\theta_1$, given $\delta=5\%$ tolerance for failure, are shown alongside the true RRF approximation error, performed on $100\times100$ matrices with either slow-decaying singular values ($\sigma_j=j^{-2}$) or fast-decaying singular values ($\sigma_j=2^{-j}$), for various target ranks $k$ and oversampling $p=1$. Empirical approximation errors (black) are the empirical 95th percentile from 20,000 RRF samples. Estimates were computed according to our~\cref{thm:Fgap-bound} (blue) and~\cite[Thm.~6]{Saibaba2019} (orange). {\em Bottom row}: The CDFs for $\sin\theta_1$ are computed for matrices with each singular value decay rate at various $k$. The 95th percentiles are marked by a horizontal dashed red line and black dots; they closely agree with the corresponding empirical observations in the top row plots. For the matrix with fast-decaying singular values, the CDF is plotted only for $k=8$ since the CDFs over all $1\le k\le12$ are very similar.}
\label{fig:bound-numerics-p1}
\end{figure}

\begin{figure}
\centering
\begin{minipage}{0.49\textwidth}
\begin{overpic}[width=\textwidth]{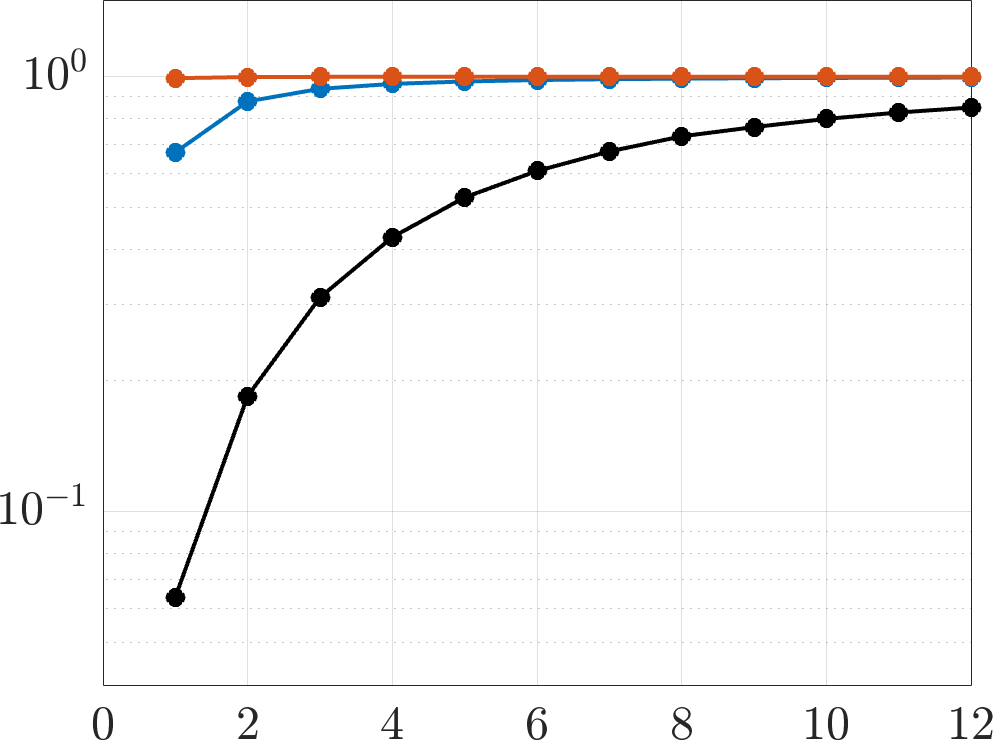}
\put(50,-5){\footnotesize $k$}
\put(20,69){\scriptsize\color{MLorange}Saibaba Thm 6}
\put(23,59){\rotatebox{10}{\scriptsize\color{MLblue}Thm 4.7}}
\put(19,34){\rotatebox{55}{\scriptsize empirical}}
\put(40,30){\parbox[l]{3cm}{\footnotesize$\sin\theta_1$ falls below black line 95\% of the time}}
\end{overpic}
\vspace*{0mm}
\end{minipage}
\hfill
\begin{minipage}{0.49\textwidth}
\begin{overpic}[width=\textwidth]{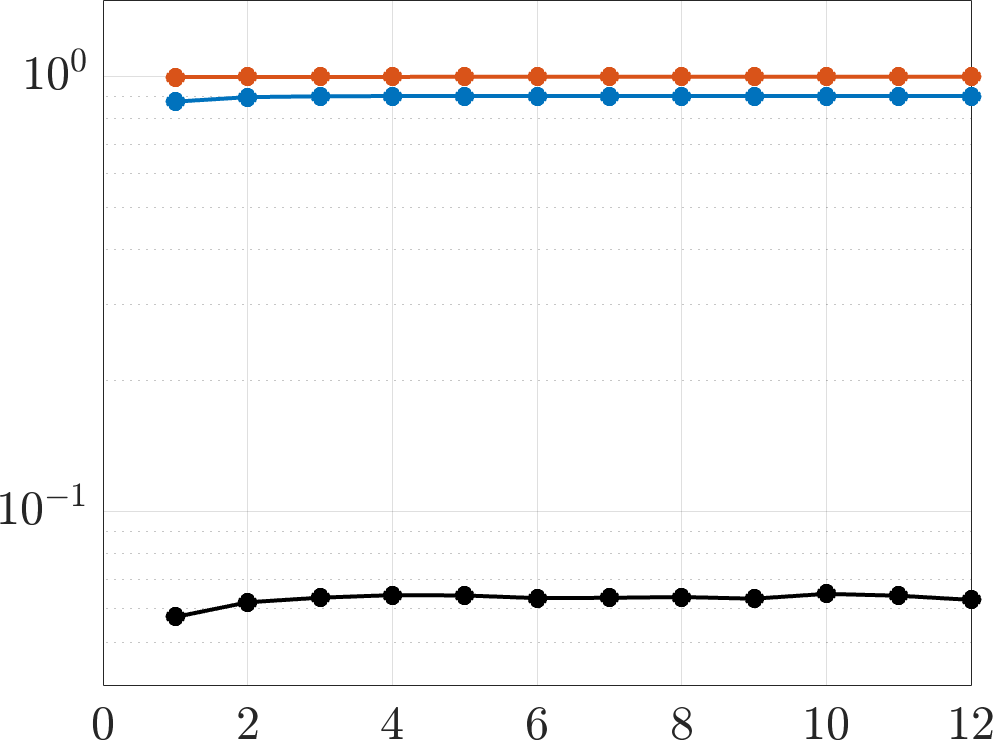}
\put(50,-5){\footnotesize $k$}
\put(20,69){\scriptsize\color{MLorange}Saibaba Thm 6}
\put(20,60){\scriptsize\color{MLblue}Thm 4.7}
\put(30,17){\scriptsize empirical}
\put(40,30){\parbox[l]{3cm}{\footnotesize$\sin\theta_1$ falls below black line 95\% of the time}}
\end{overpic}
\vspace*{0mm}
\end{minipage}
\\\vspace{1mm}
\hspace{0.7mm}
\begin{minipage}{0.475\textwidth}
\begin{overpic}[width=\textwidth]{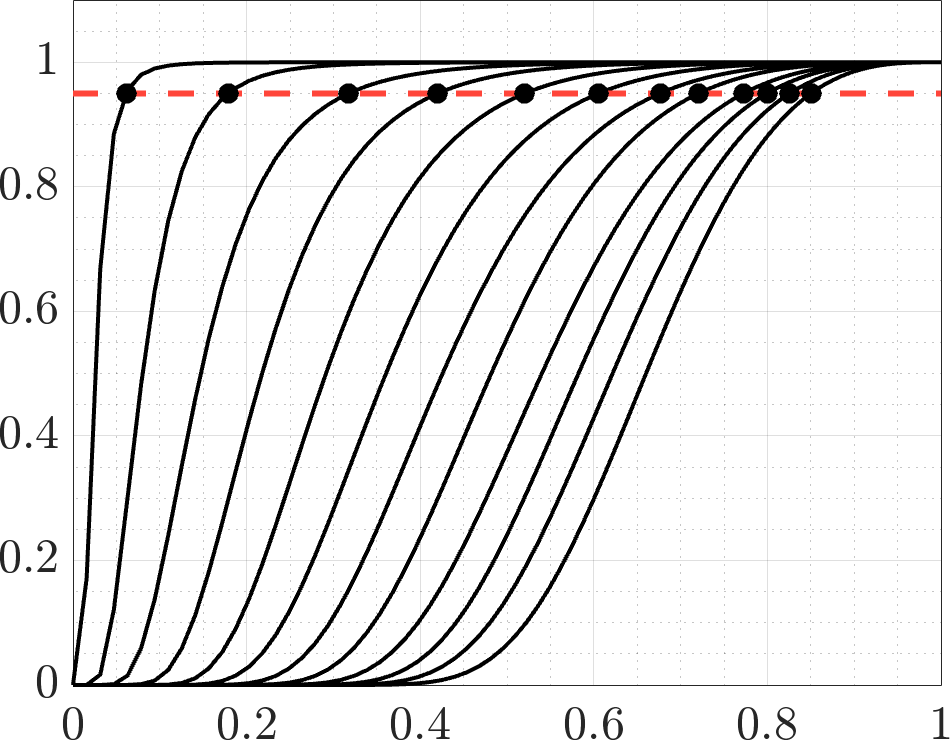}
\put(50,-5){\footnotesize $\sin\theta$}
\put(64,73){\scriptsize exact CDFs~\cref{eq:cdf}}
\put(89,63.5){\scriptsize\color{myred}$95\%$}
\put(11.5,50){\rotatebox{85}{\scriptsize $k=1$}}
\put(17.5,48){\rotatebox{80}{\scriptsize $k=2$}}
\put(23.8,45){\rotatebox{73}{\scriptsize $k=3$}}
\put(29,40){\rotatebox{70}{\scriptsize $k=4$}}
\put(34.5,36){\rotatebox{69}{\scriptsize $k=5$}}
\put(63,22){\rotatebox{68}{\scriptsize $k=12$}}
\end{overpic}
\vspace*{0mm}
\subcaption{Slow-decaying singular values, $p=5$.}
\end{minipage}
\hfill
\begin{minipage}{0.475\textwidth}
\begin{overpic}[width=\textwidth]{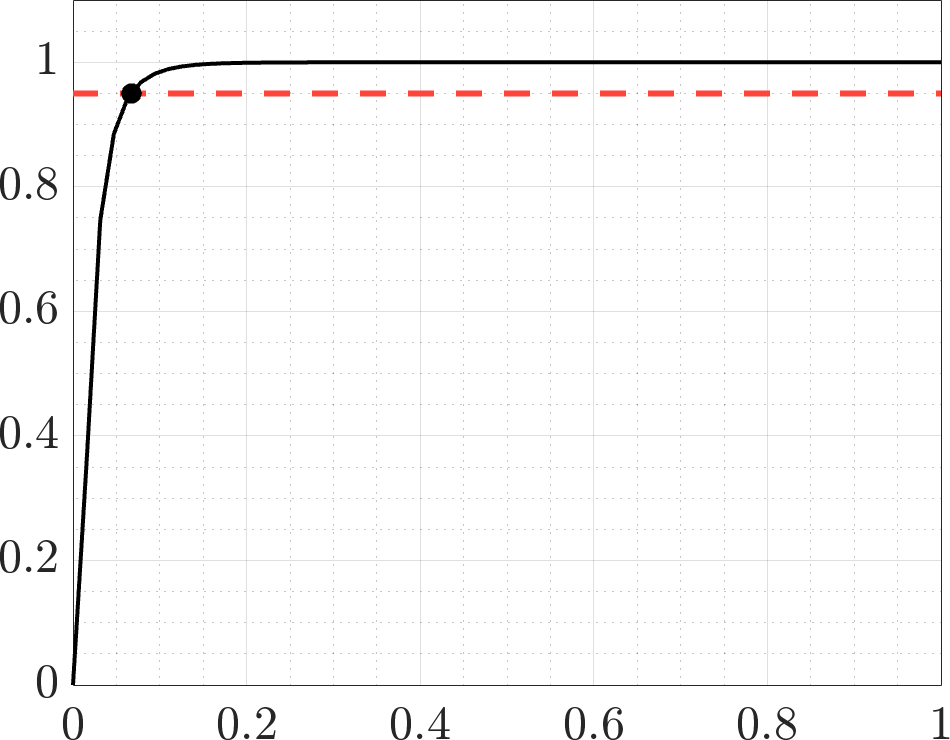}
\put(50,-5){\footnotesize $\sin\theta$}
\put(64,73){\scriptsize exact CDF~\cref{eq:cdf}}
\put(89,63.5){\scriptsize\color{myred}$95\%$}
\put(11,40){\rotatebox{88}{\scriptsize $k=8$}}
\end{overpic}
\vspace*{0mm}
\subcaption{Fast-decaying singular values, $p=5$.}
\end{minipage}
\caption{Analogous plots as in~\cref{fig:bound-numerics-p1} are shown, but with oversampling $p=5$.}
\label{fig:bound-numerics-p5}
\end{figure}

\subsection{Numerical experiments}\label{ss:bound-numerics}
Here, we present numerical experiments for our error estimates in~\cref{thm:Fgap-bound,thm:gap-bound}, in comparison with~\cite{Saibaba2019}.

\Cref{fig:bound-numerics-p1,fig:bound-numerics-p5} show error estimates with respect to a 95\% guarantee of success for $\sin\theta_1$, where $\theta_1$ is the largest principal angle between the true dominant left singular subspace of a matrix $A\in\R^{100\times100}$ and the RRF approximation, for various choices of target rank $k$, oversampling $p$, and rates of singular value decay. We compare the estimates of~\cref{thm:Fgap-bound,thm:gap-bound} with that of~\cite[Thm.~6]{Saibaba2019}. 

We make the following observations:
\begin{enumerate}[noitemsep,leftmargin=*]
\item In none of the four cases (slow or fast singular value decay, and $p=1$ or 5) does the estimate of Saibaba~\cite{Saibaba2019} give meaningful bounds, due to the large singular value ratios. Our estimate from~\cref{thm:Fgap-bound} gives an excellent bound when $p=1$; it is considerably weaker for $p=5$, although it remains tighter than that of~\cite{Saibaba2019}.\footnote{\,Of course, Saibaba's estimates become much tighter with a few subspace iterations, but the extension of our analysis to incorporate subspace iterations still demonstrates an improvement over Saibaba's bounds (see~\cref{ss:iterative}).}
\item When the singular values are known, one can numerically compute the exact bound on $\sin\theta_1$ that guarantees a $1-\delta$ probability of success by inverting the corresponding CDF from~\cref{eq:cdf}.
\item Asymptotic estimates, like those of~\cite{Dong2024}, fail to provide true upper bounds on the approximation error in practice because they require a somewhat arbitrary choice for the constants appearing in the asymptotic terms.  Conversely, our estimates are explicit upper bounds and provide guarantees for all choices of $k$, $p$, and $\delta$.
\end{enumerate}
In summary, our bounds provide rigorous and more meaningful guarantees for the RRF approximation error that are applicable for any choice of target rank $k$, oversampling $p$, and failure tolerance $\delta$, even in the absence of large singular value gaps.

\section{Discussion}\label{s:discussion}
In this section, we discuss a conjecture for improved constants in our error estimates, an alternative method for estimating upper bounds, and extensions to iterative algorithms like randomized subspace iteration.

\subsection{Improved constants with the Frobenius singular value ratio}
We believe that the constant $C_0$ in our estimate in~\Cref{thm:Fgap-bound} can be replaced by the improved constants of~\Cref{prop:scalar-bd}, depending on the parity of $p$. Indeed, the latter constants feature a factor of $\delta^{-1/p}$ or $\delta^{-1/(p+1)}$ rather than $\delta^{-1/2}$, which can be significant with larger amounts of oversampling; compare our excellent bounds in~\cref{fig:bound-numerics-p1} against the much weaker bounds in~\cref{fig:bound-numerics-p5}. This discrepancy is especially notable for the single-pass RRF, in which no subspace iterations are employed.

A heuristic argument for why we believe this improvement holds is suggested by the proof of~\cref{thm:Fgap-bound}. In essence, the proof reduces the situation with general $p\ge0$ to the case of $p=0,1$, where one then uses the quasiconvexity of the beta type II density function~\cref{eq:beta-ii}. For general $p>1$, the corresponding beta type II density is no longer quasiconvex, but when the size of the matrix $N$ is much larger than the oversampling parameter $p$,  which is true in most practical scenarios,  the density~\cref{eq:beta-ii} is {\em almost} quasiconvex. If one can appropriately quantify the notion of ``almost quasiconvex,'' then the proof should be adaptable to the $p>1$ case with the improved constant from~\Cref{prop:scalar-bd}. We state our conjecture formally as follows; observe that~\cref{thm:Fgap-bound} is exactly the conjecture when $p=0,1$.

\begin{conjecture}\label[conjecture]{conj:Fgap}
Assume the hypotheses of~\cref{thm:cdf} with $\rank(A)>k+p$ and $N\ge2k+p$, and let $\xi_k$ be the Frobenius singular value ratio~\cref{eq:Fgap}. Then $\sin\theta_1$ satisfies the bounds~\cref{eq:scalar-p1-odd,eq:scalar-p1-even,eq:scalar-p0}, replacing $b/a=\xi_k$, with probability $\ge1-\delta$.
\end{conjecture}
We compare these conjectured bounds with the bound of~\cref{thm:Fgap-bound}, as well as another approach described in~\cref{ss:cdf-estimate} (see~\cref{fig:bound-numerics-conj}).

\subsection{CDF-based bounds from estimated singular values}\label{ss:cdf-estimate}
Often in practice, the singular values of the matrix $A$ are unknown, which makes directly computing the CDF~\cref{eq:cdf} impossible. However, one can still obtain approximate bounds using a method suggested by~\cite{Dong2024}, in which one estimates the RRF error via the exact formula~\cref{eq:cdf} for the CDF of the error using approximate singular values. Namely, we approximate the leading singular values by $\hat\sigma_1,\dots,\hat\sigma_{k+p}$ via the randomized SVD (RSVD)~\cite{Halko2011}; then, we pad the remaining singular values as $\hat\sigma_j=\hat\sigma_{k+p}$ for $k+p+1\le j\le N$. Lastly, we insert these approximate singular values into~\cref{eq:cdf} to compute the 95th percentile. One can also incorporate power iterations to improve the singular value accuracy. % Of course, if the true singular values $\sigma_1,\dots,\sigma_N$ were known in advance, then we would recover exact probabilistic guarantees for the RRF approximation error. 

We observe in~\cref{fig:bound-numerics-conj} that such a CDF-based method, in which one combines the formula~\cref{eq:cdf} with singular values approximated via the RSVD, can yield practical bounds that are comparable or even better than the bounds in~\cref{conj:Fgap}.

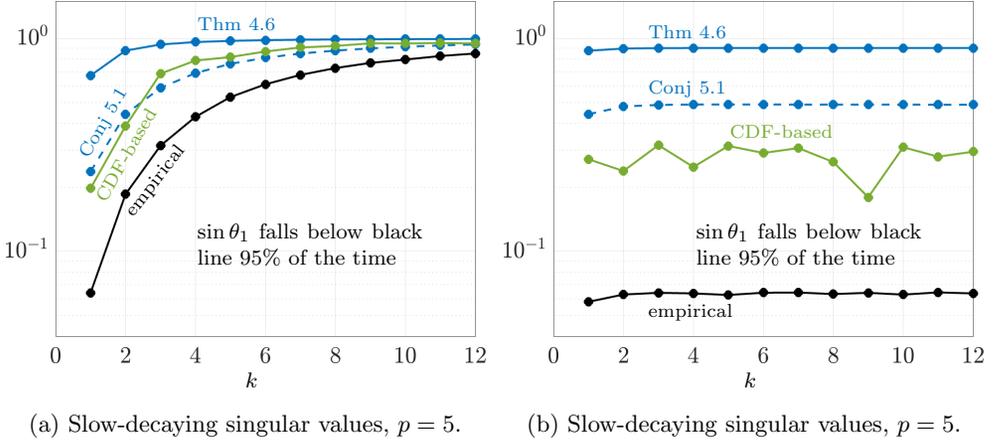
\begin{figure}
\centering
\begin{minipage}{0.49\textwidth}
\begin{overpic}[width=\textwidth]{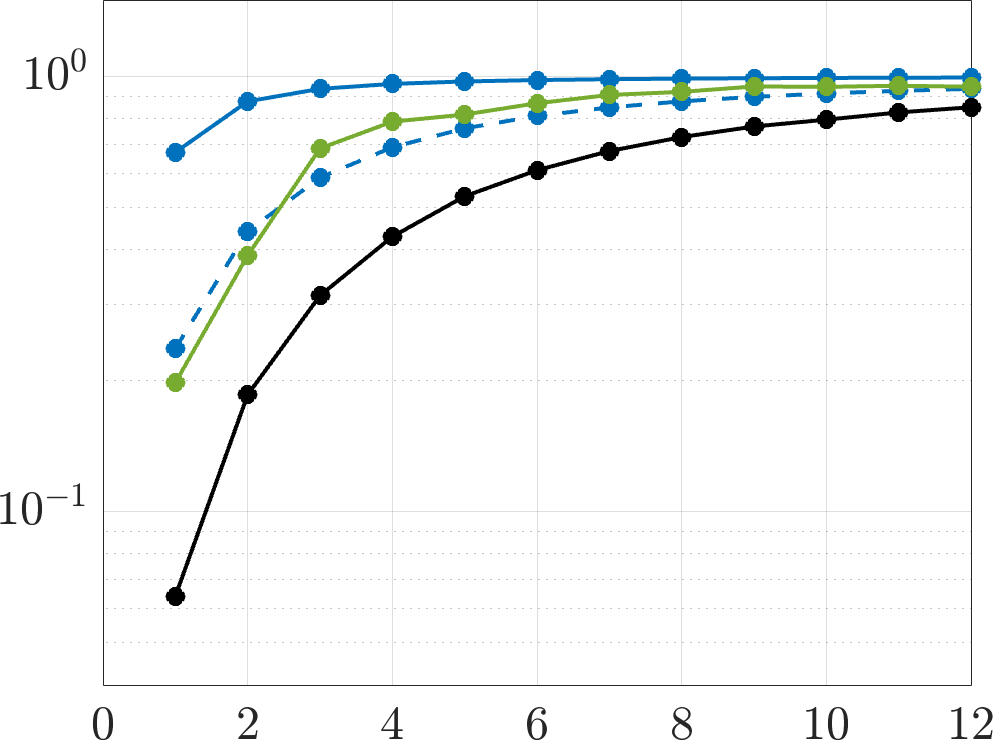}
\put(50,-5){\footnotesize $k$}
\put(40,68.5){\scriptsize\color{MLblue}Thm 4.7}
\put(15,42){\rotatebox{60}{\scriptsize\color{MLblue}Conj 5.1}}
\put(18.5,33){\rotatebox{60}{\scriptsize\color{MLgreen} CDF-based}}
\put(24.5,30){\rotatebox{52}{\scriptsize empirical}}
\put(40,23){\parbox[l]{3cm}{\footnotesize$\sin\theta_1$ falls below black line 95\% of the time}}
\end{overpic}
\vspace*{0mm}
\subcaption{Slow-decaying singular values, $p=5$.}
\end{minipage}
\hfill
\begin{minipage}{0.49\textwidth}
\begin{overpic}[width=\textwidth]{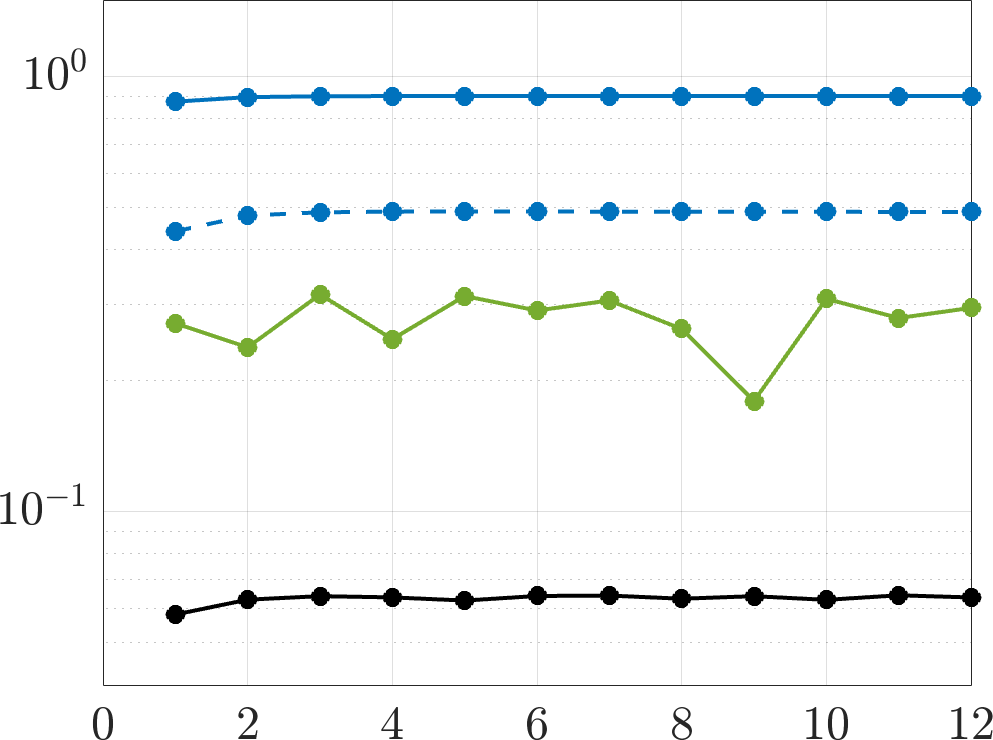}
\put(50,-5){\footnotesize $k$}
\put(30,67){\scriptsize\color{MLblue}Thm 4.7}
\put(30,55.5){\scriptsize\color{MLblue}Conj 5.1}
\put(47,46.5){\scriptsize\color{MLgreen} CDF-based}
\put(30,9.5){\scriptsize empirical}
\put(40,23){\parbox[l]{3cm}{\footnotesize$\sin\theta_1$ falls below black line 95\% of the time}}
\end{overpic}
\vspace*{0mm}
\subcaption{Fast-decaying singular values, $p=5$.}
\end{minipage}
\caption{Analogous plots as in~\cref{fig:bound-numerics-p1} are shown, but with oversampling $p=5$ and different estimates. Estimates were computed according to our~\cref{thm:Fgap-bound} (solid blue),~\cref{conj:Fgap} (dashed blue), and the CDF-based method using approximate singular values (see~\cref{ss:cdf-estimate}). Plots for $p=1$ are omitted, as~\cref{thm:Fgap-bound} and~\cref{conj:Fgap} are equivalent for $p=0,1$.}
\label{fig:bound-numerics-conj}
\end{figure}

\subsection{Iterative algorithms}\label{ss:iterative}

The analysis in this work extends to iterative versions of the RRF, namely, randomized subspace iteration (RSI) and randomized block Krylov iteration (RBKI)~\cite{Musco2015}. Mathematically, the approximate singular subspaces of RSI are equivalent to those obtained by applying the RRF,~\cref{alg:RRF}, to the matrix $(AA^\top)^qA$, where $q$ is the number of subspace iterations. Something similar holds for the Krylov subspace generated by RBKI: out of all $(k+p)$-dimensional subspaces of the block Krylov space, the closest approximation to the true dominant left singular subspace of $A$ is at least as good as the RRF approximation to the corresponding (not necessarily dominant) left singular subspace of the matrix $\phi(A)$, for any polynomial $\phi$ with degree $\le 2q+1$ and only odd-degree terms that do not vanish on the dominant $k$ singular values. Here, $\phi(A)$ denotes the matrix with the same singular vectors as $A$ but with $|\phi|$ applied to each of $A$'s singular values. Therefore, analogues of~\cref{thm:cdf,thm:Fgap-bound,thm:gap-bound} follow immediately for RSI and RBKI; for RBKI, the results pertain to the angle between the true dominant singular subspace of $A$ and the closest $(k+p)$-dimensional subspace within the Krylov space. One simply replaces the singular values $\Sigma$ in those theorem statements with $\Sigma^{2q+1}$ and $\phi(\Sigma)$ for RSI and RBKI respectively; the equality in~\cref{thm:cdf} becomes an inequality for RBKI since any particular choice of $\phi$ may not be the best polynomial in the Krylov subspace.\footnote{\,We note that our analysis also directly extends to the case where $\phi$ is a rational function, which is useful when one has access to fast linear solves with $A$. Rational $\phi$'s appear in, e.g., the FEAST family of eigensolvers~\cite{Tang2014}.}

Generically, the subspace approximation improves exponentially in the number of iterations for both RSI and RBKI. In the absence of a singular value gap, however, gap-dependent bounds such as~\cite[Thm.~6]{Saibaba2019} do not capture this phenomenon and, in fact, can fail to tell the difference between fewer and more subspace iterations. Our bounds in terms of the Frobenius singular value ratio better reflect the convergence behavior of RSI in no-gap and small-gap regimes. On the other hand, there are a number of fully gap-independent analyses of these two iterative algorithms which aim to estimate the number of RSI or RBKI iterations needed to remain within some user-prescribed error tolerance for the subspace approximation task~\cite{AllenZhu2016,AllenZhu2017,Meyer2024,Musco2015}. These estimates also tend to be overly pessimistic in terms of the number of iterations necessary since they incorporate no spectral information from the matrix at all. We believe our results can be used to produce better bounds on the number of RSI or RBKI iterations required to achieve a certain error tolerance, assuming one knows the total Frobenius weights of the leading and tail singular values.

\section*{Declarations}
The original manuscript, announced on the arXiv on March 1, 2026, was prepared entirely without the assistance of AI. ChatGPT-6 Astra was used in early September 2026 while this manuscript was in revision to audit for mathematical errors; minor corrections were incorporated into the revision. The proof techniques for~\Cref{lem:J-scalar,lem:p0} were suggested in the audit by ChatGPT, as were fixes for mistakes in the proof of~\Cref{lem:monotone} and in Part II of the proof of~\Cref{thm:cdf}. The remaining content was not developed with the assistance of AI, and none of the text, figures, or code in the manuscript are AI-generated. The authors assume responsibility for all content.

\bibliographystyle{siam}
\bibliography{references}

@book{Nachbin1965,
  title={The Haar Integral},
  author={L. Nachbin},
  year={1965},
  address={Princeton},
  publisher={D. Van Nostrand Company}
}

@book{Stein2003,
  title={Complex Analysis},
  author={Stein, E. M. and Shakarchi, R.},
  year={2003},
  address={Princeton},
  publisher={Princeton University Press}
}

@book{Horn1985,
  title={Matrix Analysis},
  author={Horn, R. A. and Johnson, C. R.},
  year={1985},
  publisher={Cambridge University Press},
  address={Cambridge}
}

@book{Kato1980,
  title={Perturbation Theory for Linear Operators},
  author={Kato, T.},
  year={1980},
  publisher={Springer},
  address={Berlin/Heidelberg}
}

@book{Baker1975,
  title={Essentials of {Pad\'e} Approximation},
  author={{Baker, Jr.}, G. A.},
  year={1975},
  publisher={Academic Press},
  address={New York}
}

@book{Bhatia1997,
  title={Matrix Analysis},
  author={R. Bhatia},
  year={1997},
  address={New York},
  publisher={Springer}
}

@book{Boyd2004,
  title={Convex Optimization},
  author={Boyd, S. and Vandenberghe, L.},
  year={2004},
  address={Cambridge},
  publisher={Cambridge University Press}
}

@book{Chikuse2003,
  title={Statistics on Special Manifolds},
  author={Chikuse, Y.},
  year={2003},
  address={New York},
  publisher={Springer}
}

@book{Durrett2019,
  title={Probability: Theory and Examples},
  author={Durrett, R.},
  edition={5th},
  year={2019},
  publisher={Cambridge University Press},
  address={Cambridge}
}

@book{Forrester2010,
  title={Log-Gases and Random Matrices},
  author={Forrester, P. J.},
  year={2010},
  address={Princeton},
  publisher={Princeton University Press}
}

@book{Golub2013,
  title={Matrix Computations},
  author={Golub, G. H. and {Van Loan}, C. F.},
  year={2013},
  edition={4th},
  address={Baltimore},
  publisher={Johns Hopkins University Press}
}

@book{Gupta2000,
  title={Matrix Variate Distributions},
  author={Gupta, A. K. and Nagar, D. K.},
  year={2000},
  address={Boca Raton, FL},
  publisher={Chapman \& Hall/CRC}
}

@book{Macdonald1995,
  title={Symmetric Functions and Hall Polynomials},
  author={Macdonald, I. G.},
  edition={2nd},
  publisher={Oxford University Press},
  address={Oxford},
  year={1995}
}

@book{Muirhead1982,
  title={Aspects of Multivariate Statistical Theory},
  author={R. J. Muirhead},
  year={1982},
  address={Hoboken, NJ},
  publisher={John Wiley \& Sons}
}

@book{Shaked2007,
  title={Stochastic Orders},
  author={Shaked, M. and Shanthikumar, J. G.},
  year={2007},
  address={New York},
  publisher={Springer}
}

@inproceedings{AllenZhu2016,
  title={{LazySVD}: Even faster {SVD} decomposition yet without agonizing pain},
  author={Allen-Zhu, Z. and Li, Y.},
  booktitle={Advances in Neural Information Processing Systems},
  volume={29},
  year={2016},
  address={Barcelona}
}

@inproceedings{AllenZhu2017,
  title={First efficient convergence for streaming {$k$-PCA}: A global, gap-free, and near-optimal rate},
  author={Allen-Zhu, Z. and Li, Y.},
  booktitle={IEEE 58th Annual Symposium on Foundations of Computer Science},
  pages={487--492},
  year={2017},
  address={Berkeley}
}

@article{Absil2006,
  title={On the largest principal angle between random subspaces},
  author={Absil, P.-A. and Edelman, A. and Koev, P.},
  journal={Linear Algebra Appl.},
  volume={414},
  pages={288--294},
  year={2006}
}

@article{Alimisis2024,
  title={Geodesic Convexity of the Symmetric Eigenvalue Problem and Convergence of {Ri}emannian Steepest Descent},
  author={F. Alimisis and B. Vandereycken},
  journal={J. Optim. Theory Appl.},
  volume={203},
  pages={920--959},
  year={2024}
}

@article{Armstrong2025,
  title={Structure-Aware Analyses and Algorithms for Interpolative Decompositions},
  author={R. Armstrong and A. Buzali and A. Damle},
  journal={SIAM J. Sci. Comput.},
  volume={47},
  number={3},
  pages={A1527--A1554},
  year={2025}
}

@article{Avron2010,
  title={Blendenpik: Supercharging {LAPACK}'s least-squares solver},
  author={Avron, H. and Maymounkov, P. and Toledo, S.},
  journal={SIAM J. Sci. Comput.},
  volume={32},
  number={3},
  pages={1217--1236},
  year={2010},
  publisher={SIAM}
}

@article{Baker1997,
  title={The {Calogero--Sutherland} model and generalized classical polynomials},
  author={Baker, T. H. and Forrester, P. J.},
  journal={Comm. Math. Phys.},
  volume={188},
  pages={175--216},
  year={1997}
}

@inproceedings{Balcan2016,
  title={An improved gap-dependency analysis of the noisy power method},
  author={Balcan, M.-F. and Du, S. S. and Wang, Y. and Yu, A. W.},
  booktitle={29th Annual Conference on Learning Theory},
  pages={284--309},
  volume={49},
  year={2016},
  address={New York},
  organization={PMLR}
}

@article{Beyn2012,
  title={An integral method for solving nonlinear eigenvalue problems},
  author={Beyn, W.-J.},
  journal={Linear Algebra Appl.},
  volume={436},
  number={10},
  pages={3839--3863},
  year={2012},
  publisher={Elsevier}
}

@inproceedings{Boutsidis2015,
  title={Spectral clustering via the power method - provably},
  author={Boutsidis, C. and Kambadur, P. and Gittens, A.},
  booktitle={Proceedings of the 32nd International Conference on Machine Learning},
  pages={40--48},
  year={2015},
  volume = {37},
  address = {Lille, France}
}

@article{DasGupta1976,
  title={A generalization of {A}nderson's theorem on unimodal functions},
  author={{Das Gupta}, S.},
  journal={Proc. Amer. Math. Soc.},
  volume={60},
  number={1},
  pages={85--91},
  year={1976}
}

@article{Davis1970,
  title={The rotation of eigenvectors by a perturbation {III}},
  author={C. Davis and W. M. Kahan},
  journal={SIAM J. Numer. Anal.},
  volume={7},
  pages={1--46},
  year={1970}
}

@article{DiazGarcia1997,
  title={Proof of the conjectures of {H.~Uhlig} on the singular multivariate beta and the {J}acobian of a certain matrix transformation},
  author={D{\'i}az-Garc{\'i}a, J. and Guti{\'e}rrez J{\'a}imez, R.},
  journal={Ann. Stat.},
  volume={25},
  number={5},
  pages={2018--2023},
  year={1997}
}

@article{DiazGarcia1997b,
  title={Wishart and pseudo-{W}ishart distributions and some applications to shape theory},
  author={D{\'i}az-Garc{\'i}a, J. A. and Guti{\'e}rrez J{\'a}imez, R. and Mardia, K. V.},
  journal={J. Multivariate Anal.},
  volume={63},
  pages={73--87},
  year={1997}
}

@article{DiazGarcia2007,
  title={A note about measures and {J}acobians of singular random matrices},
  author={D{\'i}az-Garc{\'i}a, J.},
  journal={J. Multivariate Anal.},
  volume={98},
  pages={960--969},
  year={2007}
}

@article{DiazGarcia2009,
  title={Jacobians of certain transformations of singular matrices},
  author={D{\'i}az-Garc{\'i}a, J. and Guti{\'e}rrez-J{\'a}imez, R.},
  journal={Appl. Math. (Warsaw)},
  volume={36},
  number={2},
  pages={241--249},
  year={2009}
}

@article{Dong2024,
  title={Efficient bounds and estimates for canonical angles in randomized subspace approximations},
  author={Dong, Y. and Martinsson, P.-G. and Nakatsukasa, Y.},
  journal={SIAM J. Matrix Anal. Appl.},
  volume={45},
  number={4},
  pages={1978--2006},
  year={2024}
}

@article{Drineas2012,
  title={Fast approximation of matrix coherence and statistical leverage},
  author={Drineas, P. and Magdon-Ismail, M. and Mahoney, M. W. and Woodruff, D. P.},
  journal={J. Mach. Learn. Res.},
  volume={13},
  number={1},
  pages={3475--3506},
  year={2012}
}

@article{Gross1989,
  title={Total positivity, spherical series, and hypergeometric functions of matrix argument},
  author={Gross, K. I. and Richards, D. St. P.},
  journal={J. Approx. Theory},
  volume={59},
  pages={224--246},
  year={1989}
}

@article{Halko2011,
  title={Finding structure with randomness: Probabilistic algorithms for constructing approximate matrix decompositions},
  author={Halko, N. and Martinsson, P.-G. and Tropp, J. A.},
  journal={SIAM Rev.},
  volume={53},
  number={2},
  pages={217--288},
  year={2011},
  publisher={SIAM}
}

@article{Hayakawa1966,
  title={On the distribution of a quadratic form in a multivariate normal sample},
  author={Hayakawa, T.},
  journal={Ann. Inst. Statist. Math.},
  volume={18},
  pages={191--201},
  year={1966},
  publisher={Springer}
}

@article{Herz1955,
  title={Bessel functions of matrix argument},
  author={Herz, C. S.},
  journal={Ann. Math.},
  volume={61},
  number={3},
  pages={474--523},
  year={1955}
}

@article{Horning2022,
  title={Twice is enough for dangerous eigenvalues},
  author={Horning, A. and Nakatsukasa, Y.},
  journal={SIAM J. Matrix Anal. Appl.},
  volume={43},
  number={1},
  pages={68--93},
  year={2022},
  publisher={SIAM}
}

@article{Hsu1940,
  title={An algebraic derivation of the distribution of rectangular coordinates},
  author={Hsu, P.-L.},
  journal={Proc. Edinb. Math. Soc.},
  volume={6},
  year={1940},
  pages={185--189}
}

@article{James1968,
  title={Calculation of Zonal Polynomial Coefficients by Use of the {Laplace--Beltrami} Operator},
  author={James, A. T.},
  journal={Ann. Math. Stat.},
  volume={39},
  number={5},
  pages={1711--1718},
  year={1968}
}

@article{Jiu2020,
  title={Calculation and Properties of Zonal Polynomials},
  author={Jiu, L. and Koutschan, C.},
  journal={Math. Comput. Sci.},
  volume={14},
  pages={623--640},
  year={2020}
}

@article{Koev2006,
  title={The efficient evaluation of the hypergeometric function of a matrix argument},
  author={Koev, P. and Edelman, A.},
  journal={Math. Comp.},
  volume={75},
  number={254},
  pages={833--846},
  year={2006}
}

@article{Massey2024,
  title={Admissible subspaces and the subspace iteration method},
  author={Massey, P.},
  journal={BIT},
  volume={64},
  number={12},
  year={2024}
}

@article{Massey2025,
  title={Dominant subspace and low-rank approximations from block {K}rylov subspaces without a prescribed gap},
  author={Massey, P.},
  journal={Linear Algebra Appl.},
  volume={708},
  pages={112--149},
  year={2025},
  publisher={Elsevier}
}

@inproceedings{Meyer2024,
  title={On the unreasonable effectiveness of single vector {K}rylov methods for low-rank approximation},
  author={Meyer, R. and Musco, C. and Musco, C.},
  booktitle={Proceedings of the 2024 Annual ACM--SIAM Symposium on Discrete Algorithms},
  pages={811--845},
  year={2024},
  address={Alexandria, VA},
  organization={SIAM}
}

@inproceedings{Musco2015,
  title={Randomized block {K}rylov methods for stronger and faster approximate singular value decomposition},
  author={Musco, C. and Musco, C.},
  booktitle={Advances in Neural Information Processing Systems},
  volume={28},
  year={2015},
  address={Montreal}
}

@article{Nakatsukasa2020,
  title={Sharp error bounds for {R}itz vectors and approximate singular vectors},
  author={Nakatsukasa, Y.},
  journal={Math. Comp.},
  volume={89},
  number={324},
  pages={1843--1866},
  year={2020}
}

@article{Richards2024,
  title={A reflection formula for the {G}aussian hypergeometric function of matrix argument},
  author={Richards, D. and Zheng, Q.},
  journal={J. Math. Anal. Appl.},
  volume={531},
  number={2},
  pages={127862},
  year={2024}
}

@article{Saibaba2019,
  title={Randomized subspace iteration: Analysis of canonical angles and unitarily invariant norms},
  author={Saibaba, A. K.},
  journal={SIAM J. Matrix Anal. Appl.},
  volume={40},
  number={1},
  pages={23--48},
  year={2019}
}

@article{Shimizu2021,
  title={Heterogeneous hypergeometric functions with two matrix arguments and the exact distribution of the largest eigenvalue of a singular beta-{Wishart} matrix},
  author={Shimizu, K. and Hashiguchi, H.},
  journal={J. Multivariate Anal.},
  volume={183},
  pages={104714},
  year={2021}
}

@article{Uhlig1994,
  title={On singular {W}ishart and singular multivariate beta distributions},
  author={Uhlig, H.},
  journal={Ann. Stat.},
  volume={22},
  number={1},
  pages={395--405},
  year={1994}
}

@article{Wedin1972,
  title={Perturbation bounds in connection with singular value decomposition},
  author={P.-{\AA}. Wedin},
  journal={BIT},
  volume={12},
  pages={99--111},
  year={1972}
}

@article{Xu2020,
  title={A unified linear convergence analysis of {$k$-SVD}},
  author={Xu, Z. and Ke, Y. and Cao, X. and Zhou, C. and Wei, P. and Gao, X.},
  journal={Memetic Comput.},
  volume={12},
  pages={343--353},
  year={2020}
}

@misc{NIST,
  key={NIST DLMF},
  title={{\it NIST Digital Library of Mathematical Functions}},
  howpublished={Release 1.2.5 of 2025-12-15},
  note={{F.~W.~J.} Olver, A.~B. {Olde Daalhuis}, D.~W. Lozier, B.~I. Schneider, R.~F. Boisvert, C.~W. Clark, B.~R. Miller, B.~V. Saunders, H.~S. Cohl, and M.~A. McClain, eds.}
}

@article{Woolfe2008,
  title={A fast randomized algorithm for the approximation of matrices},
  author={F. Woolfe and E. Liberty and V. Rokhlin and M. Tygert},
  journal={Appl. Comput. Harmon. Anal.},
  year={2008},
  volume={25},
  pages={335--366}
}

@article{Martinsson2011,
  title={A randomized algorithm for the decomposition of matrices},
  author={P.-G. Martinsson and V. Rokhlin and M. Tygert},
  journal={Appl. Comput. Harmon. Anal.},
  year={2011},
  pages={47--68},
  volume={30}
}

@article{Tang2014,
  title={{FEAST} as a subspace iteration eigensolver accelerated by approximate spectral projection},
  author={P. T. P. Tang and E. Polizzi},
  journal={SIAM J. Matrix Anal. Appl.},
  volume={35},
  number={2},
  pages={354--390},
  year={2014},
  publisher={SIAM}
}

@article{Absil2004,
  title={Riemannian geometry of {G}rassmann manifolds with a view on algorithmic computation},
  author={Absil, P.-A. and Mahony, R. and Sepulchre, R.},
  journal={Acta Appl. Math.},
  volume={80},
  number={2},
  pages={199--220},
  year={2004},
  publisher={Springer}
}

@article{Frankl1990,
  title={Some geometric applications of the beta distribution},
  author={Frankl, Peter and Maehara, Hiroshi},
  journal={Ann. Inst. Statist. Math.},
  volume={42},
  number={3},
  pages={463--474},
  year={1990},
  publisher={Springer}
}

@article{Mulder2018,
  title={The {matrix-F} prior for estimating and testing covariance matrices},
  author={Mulder, J. and Pericchi, L. R.},
  journal={Bayesian Anal.},
  volume={13},
  number={4},
  pages={1193--1214},
  year={2018}
}
\end{document}